\documentclass[reqno]{amsart} 

\usepackage[utf8]{inputenc} 

\usepackage[foot]{amsaddr} 

\usepackage{amsthm}
\usepackage{amsmath}  
\usepackage{amssymb} 
\usepackage{amsfonts}
\usepackage{latexsym}
\usepackage{mathtools} 

\usepackage[english]{babel} 

\usepackage{cite}

\usepackage[utf8]{inputenc} 
\usepackage{textcomp} 

\usepackage{graphicx}  

\usepackage{flafter}  

\usepackage{bm}  

\usepackage[usenames,dvipsnames]{xcolor} 
    \usepackage[backref=none]{hyperref} 

\usepackage{esint}

\usepackage{calrsfs}

\theoremstyle{plain}
\newtheorem{theorem}{Theorem}
\newtheorem{lemma}[theorem]{Lemma}

\newtheorem{proposition}[theorem]{Proposition}

\numberwithin{equation}{section}
\numberwithin{theorem}{section}

 \usepackage{todonotes}
\newcommand{\eqdef }{\overset{\mbox{\tiny{def}}}{=}}
\newcommand{\rth}{{\mathbb{R}^3}}
\newcommand{\rfo}{{\mathbb{R}^4}}

\newcommand{\pv}{p}
\newcommand{\pZ}{\pv^0}
\newcommand{\pZp }{\pv'^{0}}
\newcommand{\qv}{q}
\newcommand{\qZ}{\qv^0}
\newcommand{\qZp }{\qv'^{0}}
\newcommand{\aZ}{a^0}

\newcommand{\Temp}{\vartheta}
\newcommand{\exm}{\rho}

\newcommand{\mollrV}{v_{\phi}}

\title[Propagation of uniform upper bounds for relativistic Boltzmann]{Propagation of uniform upper bounds for the spatially homogeneous relativistic Boltzmann equation}

\author[J. W. Jang]{Jin Woo Jang$^*$}
\address{$^*$Center for Geometry and Physics, Institute for Basic Science (IBS), Pohang 37673, Republic of Korea. \href{mailto:jangjinw@iam.uni-bonn.de}{jangjinw@iam.uni-bonn.de} }

\author[R. M. Strain]{Robert M. Strain$^\dagger$}
\address{$^\dagger$Department of Mathematics, University of Pennsylvania, Philadelphia, PA 19104, USA.  \href{mailto:strain@math.upenn.edu}{strain@math.upenn.edu}}

\author[S.-B. Yun]{Seok-Bae Yun$^\ddagger$}
\address{$^\ddagger$Department of mathematics, Sungkyunkwan University, Suwon 440-746, Republic of Korea. \href{mailto:sbyun01@skku.edu}{sbyun01@skku.edu} }

\begin{document}



\let\thefootnote\relax
\addtocounter{footnote}{-1}\let\thefootnote\svthefootnote

\begin{abstract}
In this paper, we prove the propagation of uniform upper bounds for the spatially homogeneous relativistic Boltzmann equation.  These polynomial and exponential $L^\infty$ bounds have been known to be a challenging open problem in relativistic kinetic theory.
To accomplish this, we establish two types of estimates for the gain part of the collision operator: first, we prove a potential type estimate and a relativistic hyper-surface integral estimate.
We then combine those estimates using the relativistic counterpart of the Carleman representation to derive uniform control of the gain term for the relativistic collision operator.  This allows us to prove the desired propagation of the uniform bounds of the solution.  We further present two applications of the propagation of the uniform upper bounds: first we give another proof of the Boltzmann $H$-theorem, and second we prove the asymptotic convergence of solutions to the relativistic Maxwellian equilibrium.
\end{abstract}

\setcounter{tocdepth}{1}

\maketitle
\tableofcontents

\thispagestyle{empty}

\section{Introduction}

The paper studies the special relativistic Boltzmann equation in the spatially homogeneous case for initial data of unrestricted size.  The Boltzmann equation including Einstein's theory of special relativity describes the statistical distribution of gaseous particles \cite{DeGroot,C-K}; it is a central dynamical model in special relativistic kinetic theory.

\subsection{Relativistic Boltzmann equation:}
The Cauchy problem for the spatially homogeneous relativistic Boltzmann equation reads
\begin{align}\label{RBE}
\begin{split}
\partial_t f= Q(f,f),\cr
f(p,0)=f_0(p),
\end{split}
\end{align}
where the particle distribution function $f(p,t)$ represents the density function of particles with momentum $p\in\mathbb{R}^3$ at time $t\geq0$.
The collision operator $Q(f,h)$ then can be decomposed as
\begin{equation*}
Q(f,h)=Q^{+}(f,h)-Q^{-}(f,h)
\end{equation*}
where the gain part $Q^+$ and the loss part $Q^-$ are defined by
\begin{align}\label{Q original}
\begin{split}
Q^+(f,h)&=\frac{1}{\pZ}\int_{\mathbb{R}^3}\frac{dq}{{\qZ }}\int_{\mathbb{R}^3}\frac{dq'\hspace{1mm}}{{\qZp }}\int_{\mathbb{R}^3}\frac{dp'\hspace{1mm}}{{\pZp }}  W(p,q|p',q')f(p')h(q'),\\
Q^-(f,h)&=\frac{1}{\pZ}\int_{\mathbb{R}^3}\frac{dq}{{\qZ }}\int_{\mathbb{R}^3}\frac{dq'\hspace{1mm}}{{\qZp }}\int_{\mathbb{R}^3}\frac{dp'\hspace{1mm}}{{\pZp }}  W(p,q|p',q')f(p)h(q).
\end{split}
\end{align}
The transition rate $W(p,q|p',q')$ is
\begin{align}\label{W}
W(p,q|p',q')=\frac{1}{2}s\sigma(g,\theta)\delta^{(4)}(p^\mu +q^\mu -p'^\mu -q'^\mu),
\end{align}
where $\sigma(g,\theta)$ is the scattering kernel measuring the interactions between particles, and the Dirac-delta function, $\delta^{(4)}$, enforces the conservation of energy and momentum \eqref{collision.invariants}.  For the sake of simplicity, and without loss of generality, we normalize several physical constants to be $1$, in particular we do not include notations for the speed of light and the rest mass.    Other notations are defined in the next section.

For later convenience, we define the \textit{collision frequency} $Lf$ as follows
\begin{align*}
Lf= \frac{1}{\pZ}\int_{\mathbb{R}^3}\frac{dq}{{\qZ }}\int_{\mathbb{R}^3}\frac{dq'\hspace{1mm}}{{\qZp }}\int_{\mathbb{R}^3}\frac{dp'\hspace{1mm}}{{\pZp }}  W(p,q|p',q')f(q),
\end{align*}
and then we rewrite  (\ref{RBE}) as
\begin{equation}\label{recall}
	\partial_t f+f Lf= Q^+(f,f).
\end{equation}
The relativistic Boltzmann operator $Q(f,f)$ satisfies (for $i=1,2,3$) that
$$
\int_\rth Q(f,f)dp=\int_\rth p^i Q(f,f)dp=\int_\rth \pZ Q(f,f)dp=0.
$$
These identities on the collision operator respectively lead to the formal conservation laws of mass, momentum, and energy respectively as follows
\begin{equation}\label{conlaw}
\int_\rth  \begin{pmatrix}1\\p\\\pZ \end{pmatrix} f(t,p)dp=\int_\rth\begin{pmatrix}1\\p\\\pZ \end{pmatrix}f_0(p) dp.
\end{equation}
The Boltzmann collision operator also formally satisfies that
$$
\int_\rth Q(f,f)\ln f dp\leq 0.
$$
This leads to the Boltzmann $H$-theorem for solutions to \eqref{RBE} which says that the entropy is non-increasing:
\begin{equation}\label{entropy.eq}
H(f(t))+\int^t_0D(f(s))ds\leq H(f_0),
\end{equation}
where the entropy functional is defined by
\begin{equation}\label{entropy.functional}
H(f(t)) = \int_\rth f(t,p)\ln f(t,p)dp.
\end{equation}
Further the entropy production rate is defined as
\begin{equation}\label{entropy.production}
D(f)=\int_{\mathbb{R}^6} \mollrV \sigma(\varrho,\theta)\big\{f(p^{\prime})f(q^{\prime})-f(p)f(q)\big\}\log\left(\frac{f(p^{\prime})f(q^{\prime})}{f(p)f(q)}\right)d\omega dpdq.
\end{equation}
This shows that the entropy functional $H(f(t))$ is decreasing in time for solutions to the relativistic Boltzmann equation \eqref{RBE}.   We will establish in Section \ref{htheorem} that the bounds proven in our main results grant sufficient control to prove that our solutions satisfy \eqref{entropy.eq}.  In the next subsection we define our notations.

%
%
%
%
%
%

\subsection{Notation} In this section we will define our various notational conventions on relativistic 4-vectors and the function spaces to be used in this article.
\begin{itemize}

\item  We use the notation $p^\mu$ where $\mu=0,1,2,3$ to denote a relativistic 4-vector.  We denote the 4-vector by it's components $p^\mu \in \{ \pZ, p^1, p^2, p^3 \}$ for $\mu\in\{0,1,2,3\}$.   Henceforth we usually call 4-vectors just vectors.

\item Generally Latin (spatial) indices $a,b,j,k,$ etc., take on the values $1,2,3,$ and Greek indices $\kappa, \lambda, \mu, \nu$, etc., take on the values $0,1,2,3$.  Indices are raised and lowered with the Minkowski metric $\eta_{\mu \nu}$ and its inverse $(\eta^{-1})^{\mu \nu}$,
such that  $p_\mu=\eta_{\mu\nu}p^\nu$.  In this article, we have that
\begin{eqnarray} \notag 
	\eta_{\mu \nu} = (\eta^{-1})^{\mu \nu}
	= \mbox{diag}(-1,1,1,1).
\end{eqnarray}

\item Here and throughout the rest of this article we use Einstein's summation convention that repeated indices, with one ``up" and one ``down" are summed over.

\item Then the Lorentz inner product of two 4-vectors with raised and lowered indices is given by
\begin{equation}\label{lorentz.inner.prd}
p^\mu q_\mu = p^\mu \eta_{\mu\nu} q^\nu=-{\pZ }{\qZ }+\sum^3_{i=1} p^i q^i.
\end{equation}

\item When a relativistic 4-vector $p^\mu$ satisfies the mass shell condition $p^\mu p_\mu=-1$ with ${\pZ }>0$, we call it
an {\bf energy-momentum} vector. In this case, we can express $p^\mu$ as $({\pZ },p)$ with $p\in \rth$.  Then ${\pZ }$, the energy of a relativistic particle with momentum $p$, is given by ${\pZ }=\sqrt{1+|p|^2}$.   In this article we always use the notation $p^\mu$ and $q^\mu$ to denote an energy-momentum vector.

Further the vectors $p^\mu$, $q^\mu$, $p^{\prime\mu}$ and $q^{\prime\mu}$ that appear in the relativistic Boltzmann equation \eqref{RBE} with \eqref{Q original} are all energy-momentum vectors.

\item We call a 4-vector $a^{\mu}$ {\bf space-like} if $a^{\mu}a_{\mu}>0$.

\item Alternatively we call $a^{\mu}$ {\bf time-like} if $a^{\mu}a_{\mu}<0$.

\item We define the weighted $L^1$ space $L^1_\exm$ with $\exm\geq 0$ as
$$L^1_\exm=L^1_\exm(\rth)=\{f:f \text{ measurable on }\rth, \hspace{1mm}\|f\|_{L^1_\exm}<\infty \},$$
where
$$
\|f\|_{L^1_\exm}\eqdef \int_{\rth}dp\hspace{1mm} (\pZ )^\exm \left| f(p) \right|.
$$
Similarly, we define the weighted $L^\infty$ space $L^\infty_{\exm}$ for $\exm\geq 0$ with the norm
$$\|f\|_{L^\infty_{\exm}}\eqdef \sup_{p\in \rth}|(\pZ )^\exm f(p)|.$$
\end{itemize}

With these notations in hand, in \eqref{W}, $s$ represents the square of the energy in the center of momentum frame
\begin{equation}
\label{s}
s=s(p^\mu,q^\mu)\eqdef-(p^\mu+q^\mu)(p_\mu+q_\mu)=2(-p^\mu q_\mu+1)\geq 0,
\end{equation}
and  $g$ denotes the relative momentum
\begin{equation}
	\label{g}
g=g(p^\mu,q^\mu)\eqdef\sqrt{(p^\mu-q^\mu)(p_\mu-q_\mu)}.  
\end{equation}
Then from \eqref{lorentz.inner.prd} we have
\begin{equation}\notag
g=g(p^\mu,q^\mu)= \sqrt{-(\pZ - \qZ)^2 + |p-q|^2},
\end{equation}
where $|p-q|$ is the standard Euclidean distance from $p$ to $q$ in three dimensional space.  In this paper we will always use $g$ and $s$ to mean $g=g(p^\mu,q^\mu)$ and $s=s(p^\mu,q^\mu)$.  However we will also use $g(a^\mu,b^\mu)$ etc for other four-vectors $a^\mu$ and $b^\mu$.  Note that $s$ and $g$ are related by $s=g^2+4$.
The scattering angle $\theta$ is defined by
\begin{equation*}
\cos\theta=\frac{(p^\mu-q^\mu)(p'_\mu-q'_\mu)}{g^2}.
\end{equation*}
This is known to be a well defined angle \cite{GL1996}, see the discussion below \eqref{p'}.  

Note that with the collision invariance
\begin{equation}\label{collision.invariants}
p^\mu + q^\mu = p^{\prime\mu} + q^{\prime\mu},  \quad \mu = 0, 1, 2, 3,
\end{equation}
we have further that  $g(p^\mu,q^\mu) = g(p^{\prime\mu},q^{\prime\mu})$ and similarly $s(p^\mu,q^\mu) = s(p^{\prime\mu},q^{\prime\mu})$.

Throughout this paper,  $C$ denotes a generic positive (generally large) uniform constant where $C$
may change values from line to line.   Further $A\lesssim B$ means that  there is a generic constant $C>0$
such that $A\leqslant CB$.  Then $A \approx B$ means that both $A\lesssim B$ and $B\lesssim A$ hold.

%
%
%
%
%

\subsection{Maxwellian equilibria} \label{SS:Maxwellians}
We now introduce the relativistic Maxwellians which are equilibria to \eqref{RBE}, they are also called the J{\"u}ttner distributions.

Given constants
$n > 0$, $\Temp > 0$, $u^{0} > 0$, and an energy-momentum vector $u^{\kappa}$ such that  $u_{\kappa} u^{\kappa} = -1$,
we define the
corresponding \emph{relativistic Maxwellian} as follows:
\begin{align}
	J = J(n,\Temp,u^\nu ;p^\mu)
	& \eqdef  \frac{n }{4 \pi k_B \Temp K_2(\frac{1}{k_B \Temp})}
		\exp\Big(\frac{p^{\kappa}u_{\kappa}}{k_B \Temp} \Big), \label{E:Maxwelliandef}
\end{align}
Above $k_B>0$ is \emph{Boltzmann's constant}, and $K_j(z)$ are the following modified second order Bessel functions:
\begin{align} \label{E:Besseldef}
K_j(z) & \eqdef \frac{(2^j)j!}{(2j)!} \frac{1}{z^j} \int_{\lambda = z}^{\lambda = \infty} e^{- \lambda}(\lambda^2 - z^2)^{j
			-  (1/2)} \, d \lambda, && (j \geq 0).
\end{align}
The relativistic Maxwellians \eqref{E:Maxwelliandef} are well known to be the global equilibrium solutions of the relativistic Boltzmann equation \eqref{RBE}; they also minimize the entropy \eqref{entropy.functional} under the restriction that their fluid proper number density $n$, their fluid temperature $\Temp$ and their fluid four-velocity $u^\mu$ are fixed (see e.g.  \cite[Chapter 2]{DeGroot} and \cite{GL1996,GS3}). 

In particular we note that since both $u^\nu$  and $p^\mu$ are future-directed (i.e. $u^0>0$ and $p^0>0$)  and timelike then we have that $p^{\kappa}u_{\kappa}<0$ since it holds that
$
p^{\kappa}u_{\kappa} \le - | p^{\kappa}p_{\kappa}|^{1/2} |u^{\kappa}u_{\kappa}|^{1/2} = -1.
$
We refer to \cite[Section 1.4]{MR2793935} for some additional explanations of the relativistic Maxwellians.

We will now define $T^{\mu \nu}[h]$, which is the \emph{energy-momentum tensor} for the relativistic Boltzmann equation, and $I^{\mu}[h],$ which is the \emph{particle current}.  Given any function $h(p)$, they are  defined as follows:
\begin{equation}
\label{E:TBoltzmanndef}
\begin{split}
T^{\mu \nu}[h] &
\eqdef
\int_{\mathbb{R}^3} p^{\mu} p^{\nu} h(p) \frac{d p}{\pZ},  \quad (0 \leq \mu,\nu \leq 3),
\\
I^{\mu}[h] &
\eqdef  \int_{\mathbb{R}^3} p^{\mu} h(p) \frac{d p}{\pZ},  \quad  (0 \leq \mu \leq 3).
\end{split}
\end{equation}
We can now express the conservation laws \eqref{conlaw} for a solution $f$ to the relativistic Boltzmann equation \eqref{RBE} as follows
\begin{equation}
\notag 
\begin{split}
T^{\mu 0}[f(t)] & =  T^{\mu 0}[f_0], \quad (0 \leq \mu \leq 3),
				\\
I^{0}[f(t)]  & =  I^{0}[f_0]  .
\end{split}
\end{equation}
Following the calculations in \cite[Proposition 3.3]{MR2793935}, it can further be shown that for the relativistic Maxwellian \eqref{E:Maxwelliandef} plugged into \eqref{E:TBoltzmanndef} we have
\begin{equation}
\notag 
\begin{split}
T^{\mu 0}[J]
& =  \left(n \frac{K_1(1/k_B\Temp)}{K_2(1/k_B\Temp)} +4k_B n \Temp \right) u^\mu u^0+ k_B n \Temp  (\eta^{-1})^{\mu 0}, \quad (0 \leq \mu \leq 3),
\\
I^{0}[J]
& =
n  u^0.
\end{split}
\end{equation}
Then for suitable initial data $f_0\ge 0$, such as those in our main theorems, we can choose constants $n > 0$, $\Temp > 0$ and  an energy-momentum vector $u^{\kappa}$ such that
\begin{equation}
\notag 
\begin{split}
T^{\mu 0}[J] & =  T^{\mu 0}[f_0], \quad (0 \leq \mu \leq 3),
				\\
I^{0}[J]  & =  I^{0}[f_0].
\end{split}
\end{equation}
This holds because there are five conservation laws and five unknowns from the constants.
Then further $u^0>0$ is defined by $u_{\kappa} u^{\kappa} = -1$.   We refer to the details of similar calculations in \cite{MR2793935,MR2989691,GL1996} and \cite[page 339]{GS3}.
This will be used in Section \ref{asympt}.

In the rest of this paper, due to the Lorentz invarance of the relativistic Boltzmann equation without loss of generality but for the sake of simplicity, we can normalize the physical constants to be one including choosing the fluid four-velocity $u^\mu$ to be $(1,0,0,0)$ with $n=\Temp=u^0=1$.  
Then, the global relativistic Maxwellian defined in \eqref{E:Maxwelliandef} is now equal to 
$$
J(p)=\frac{1}{4\pi}e^{-p^0}.
$$ 
Here we choose the Boltzmann constant $k_B=1$ as well as $n=\Temp=u^0=1$.  For the statement of Theorem \ref{maxwellian upper bound} and the proof in Section \ref{pmaxwell}, we use different temperatures, $\Temp_0$ and $\Temp_1$, while keeping the same $n=u^0=k_B=1$ so that we can make use of different global relativistic-Maxwellians such as $$C_0e^{-R_0\sqrt{1+|p|^2}}\text { and }C_1e^{-R_1\sqrt{1+|p|^2}},$$ where $R_0=\frac{1}{\Temp_0}$ and $R_1=\frac{1}{\Temp_1}$ for different $\Temp_0$ and $\Temp_1$ such that $\Temp_1>\Temp_0>0.$

\subsection{Center of momentum framework}
There are several ways to carry out the Dirac-delta integration in the collision operator (\ref{Q original}).
In the \textit{center-of-momentum} frame (or alternatively sometimes called the \textit{center-of-mass} frame) 
the gain term $Q^+$ and the loss term $Q^-$ are written  (see \cite{DeGroot} and  \cite{MR2765751}) as follows
\begin{align}\label{Q center of momentum}
\begin{split}
\displaystyle Q^+(f,h)&=\int_{\mathbb{R}^3\times \mathbb{S}^2} \mollrV \sigma(g,\theta)f(p^{\prime})h(q^{\prime})d\omega dq,\cr
\displaystyle Q^-(f,h)&=\int_{\mathbb{R}^3\times \mathbb{S}^2} \mollrV \sigma(g,\theta)f(p)h(q)d\omega dq,
\end{split}
\end{align}
where $\sigma(g,\theta)$ is again the scattering kernel, and the M{\o}ller velocity $\mollrV$ is given by
\begin{equation}\label{mollerV}
\mollrV =\mollrV(p,q)=\sqrt{\left|\frac{p}{\pZ }-\frac{q}{\qZ }\right|^2-\left|\frac{p}{\pZ }\times\frac{q}{\qZ }\right|^2}
=\frac{g\sqrt{s}}{\pZ \qZ }.
\end{equation}
Now the pre-collisional momentum pair $(p,q)$ and the post-collisional momentum pair $(p^{\prime},q^{\prime})$ are related by
\begin{equation}\label{p'}
\begin{split}
p^{\prime}=&\frac{p+q}{2}+\frac{g}{2}\left(\omega+(\gamma-1)(p+q)\frac{(p+q)\cdot\omega}{|p+q|^2}\right),\\
q^{\prime}=&\frac{p+q}{2}-\frac{g}{2}\left(\omega+(\gamma-1)(p+q)\frac{(p+q)\cdot\omega}{|p+q|^2}\right),
\end{split}\end{equation}
where $\gamma=(\pZ +\qZ )/\sqrt{s}$. The microscopic energy is given \cite{MR2765751} by
\begin{equation*}
\pZp =\frac{\pZ +\qZ }{2}+\frac{g}{2\sqrt{s}}\omega\cdot(p+q),
\quad \qZp=\frac{\pZ +\qZ }{2}-\frac{g}{2\sqrt{s}}\omega\cdot(p+q).
\end{equation*}
The relationship between the unitary $\omega \in \mathbb{S}^2$ introduced in \eqref{p'} and the scattering angle $\theta$ introduced in the scattering kernel $\sigma(g,\theta)$ is given by
$$\cos\theta = \frac{(p^\mu-q^\mu)(p'_\mu-q'_\mu)}{g^2 }= \frac{k}{|k|}\cdot \omega,$$ where $k$ is defined as $$k=-\frac{p+q}{\sqrt{s}}(p^0-q^0)+(p-q)+(\gamma-1)(p+q)\frac{(p+q)\cdot(p-q)}{|p+q|^2}.$$ The proof for this identity is given in \cite[page 5-6]{MR2765751}.
We note that the \textit{collision frequency} $Lf$ is then written by
\begin{align}\label{Lf}
\displaystyle Lf&=\int_{\mathbb{R}^3\times \mathbb{S}^2} \mollrV \sigma(g,\theta)f(q)~d\omega dq.
\end{align}
In the next section we will explain our main results.
%
%
%
%

%
%
%
%

\subsection{Main results}  In \cite[Theorem 4.2, page 933]{MR3166961}, it was shown under hypothesis \eqref{assumption} below that if
\begin{equation}\label{Existence}
f_0(p) \ge 0, \quad \|f_0\|_{L^1_2}<\infty,\quad  \int_{\mathbb{R}^3}f_0|\ln f_0|dp<\infty,
\end{equation}
then there exists a unique  global in time solution $f(p,t)\ge 0$ to \eqref{RBE} satisfying the conservation of mass, momentum and energy \eqref{conlaw}.  It was also shown that for this solution the H-theorem holds as in \eqref{entropy.eq}.   For these solutions specifically \eqref{Sol} holds.  Similar results were shown in \cite[Theorem 3.1, page 2257]{MR3169776}.

\subsubsection{Hypothesis on the collision kernel $\sigma(g,\theta)$}  We assume the relativistic analogue of the ideal hard-sphere assumption with Grad's angular cut-off assumptions.   Specifically, we assume that the collision kernel $\sigma(g,\theta)\ge 0$ satisfies
\begin{equation}\label{assumption}
\sigma(g,\theta)\approx g.
\end{equation}
We call this the ``hard ball'' as the kernel behaves as the classical Newtonian hard-sphere kernel (i.e. $|p-q|$) if either $|p+q|$ or $|p-q|$ is close to zero, or if $|p-q|$ is much larger than $|p^0-q^0|$.  We refer to \cite[Appendix B]{MR2679588} for a more detailed physical discussion of the collision kernels in relativistic kinetic theory.

Under this hypothesis, we obtain the following new $L^\infty$ propagation theorem:
\begin{theorem}{\bf [Uniform upper bound]}\label{main}
Fix $\exm>\frac{5}{2}$, and suppose that
\begin{align}\label{initial}
\|f_0\|_{L^1_1}<\infty,\quad\int_{\mathbb{R}^3}f_0|\ln f_0|dp<\infty, \quad \|f_0\|_{L^\infty_{\exm}}<\infty.
\end{align}
Let $f(p,t)\ge 0$ be a solution to the Cauchy  problem (\ref{RBE}) such that
\begin{align}\label{Sol}
\sup_{t\ge 0}\|f(t)\|_{L^1_1}<\infty,\quad\sup_{t\ge 0}\int_{\mathbb{R}^3}f(p,t)|\ln f(p,t)|dp<\infty.
\end{align}
Then $f(p,t)$ is uniformly bounded in $p$ and $t$ as follows:
\[
\sup_{t\ge 0}\|f(t)\|_{L^{\infty}} \le C_{f_0},
\]
for a constant $C_{f_0}>0$ which only depends only on the size of the initial quantities in \eqref{initial} and the conservation laws \eqref{conlaw}.
\end{theorem}

Note that conditions (\ref{initial}) and (\ref{Sol}) are the minimal requirements that we use  to prove the propagation of the $L^{\infty}$ bound.
For the current existence theory, however, a slightly more stringent condition on $f_0$, such as \eqref{Existence}, is needed to guarantee the existence of $f$ satisfying (\ref{Sol}) as in \cite[Theorem 4.2, page 933]{MR3166961}.

Also, under some additional hypothesis described below, we can further obtain the propagation of polynomial moments and the propagation of relativistic Maxwellian upper bounds in the $L^\infty$ sense as follows:

\begin{theorem} \label{maxwellian upper bound} {\bf [Polynomial and Maxwellian upper bounds]}
In addition to the assumptions of Theorem \ref{main}, suppose further that $f_0$ is bounded by a relativistic Maxwellian as:
	\[
	f_0(p)\le C_0(\pZ )^{-m_0}e^{-R_0\sqrt{1+|p|^2}}
	\]
	for some $C_0>0$, for some integer $m_0\geq 0$  and for $R_0> 0$.
	Then there exist uniform constants $C_1>0$ and $R_1 > 0$  that are  independent of $t$ such that
	\begin{equation*}
	f(p,t)\leq C_1(\pZ )^{-m_0}e^{-R_1\sqrt{1+|p|^2}} ~\mbox{ for all }t> 0.
	\end{equation*}
We remark that  $R_1 < R_0$.
\end{theorem}

Next we obtain the polynomial bounds under some slightly different assumptions.

\begin{theorem} \label{poly.upper.bound} {\bf [Polynomial upper bounds]}.
In addition to the assumptions of Theorem \ref{main}, we assume that $\| f_0\|_{L^1_{m_0+1}}<\infty$.  Suppose further that $f_0$ is bounded for some $C_0>0$ as:
	\[
	f_0(p)\le C_0(\pZ )^{-m_0-\exm}, \quad \exm >\frac{5}{2},\quad m_0\geq 0.
	\]
	Then there exist a uniform constant $C_1>0$  that is  independent of $t$ such that
	\begin{equation*}
	f(p,t)\leq C_1(\pZ )^{-m_0} ~\mbox{ for all }t\ge  0.
	\end{equation*}
\end{theorem}

Now the propagation of uniform, polynomial, and Maxwellian upper-bounds is one of the most interesting issues in the study of Boltzmann equation in that (1) it gives a control on the solution which the a-priori quantities of the Boltzmann equation, namely the conserved quantities \eqref{conlaw} and the entropy \eqref{entropy.eq}, cannot immediately provide, and (2) it is derived through the full exploitation of two important mathematical properties of the collision operator, namely, the damping effect of the loss term and the regularizing effect of the gain term.  The $L^{\infty}$ theory  also has various applications to the study of the Boltzmann equation. For example, it can be used in the proof of the H-theorem, since the $L^{\infty}$ propagation theory can guarantee that the approximate solution associated to the mollified initial data remains bounded from above and below so that the formal computation to derive the H-thereom can be justified (this is discussed in Section \ref{htheorem}). As such, a suitable $L^{\infty}$-estimate can be an important building block in the study of the asymptotic behavior of the Boltzmann equation, which is one of the most highlighted issues in the kinetic theory.

In this paper, we generalize the non-relativistic framework of \cite{MR711482,Car}. In \cite{Car},
Carleman established the uniform upper bound propagation for rotationally symmetric solutions to the classical homogeneous Boltzmann equation.  And Arkeryd in \cite{MR711482} then extended the result to general solutions without rotational symmetry.

The main idea is as follows. If one can obtain a uniform lower bound of $L$, as $Lf>C_1$ (damping effect), and a uniform upper bound for $Q^+$, as $Q^+<C_2$ (regularizing effect), for some positive constants $C_1$, $C_2$, one gets from  \eqref{recall} that
\[
\partial_tf+C_1f\leq C_2,
\]
which immediately implies the uniform boundedness of the solution.
Therefore, the key difficulty to realize this idea arises in the uniform control of $Lf$ and $Q^+$.
The relativistic adaption of these arguments, however, turned out to be highly non-trivial due to the complicated structure of the relativistic collision operator.
The lower-bound estimate of $Lf$ is already given in \cite{MR3166961}, so the main issue is whether we can obtain the uniform control on $Q^+$ as well. Applying existing known techniques for the $Q^+$ estimates from the classical Boltzmann literature such as \cite{MR711482} and \cite{Car} to the relativistic situation, however, turned out to be extremely difficult for the following reasons:
\begin{itemize}
	\item It is very limited to use the change of pre-post collisional variables $p\mapsto p'$ as the Jacobian is no longer uniformly bounded above and below in the relativistic scenario.  This was studied in \cite{Jang2016}.

	\item We have lacked a relativistic counterpart for the Carleman representation formula that we were able to use in this framework.

	\item The relativistic counterpart of the interaction hypersurface turned out to be a 2-dimensional hyperboloid: $(p'^\mu-p^\mu)(q'_\mu-p_\mu)=0$, which is highly nonlinear.

	\item Most crucially, each of the integral estimates requires extremely complicated computations due to the representations of the post-collisional momentums such as $\eqref{p'}$ and the use of the nonlinear 2-dimensional hyperboloid.    For example, estimating $Q^+(Q^+(f,g),h)$  as in \eqref{Q original}, using the prior methods,  appears to be extremely difficult in the relativistic regime due to the very complicated non-linear relativistic geometry.

\end{itemize}

In this paper, in order to resolve the  difficulties above, we derive a relativistic counterpart of the celebrated Carleman representation (Proposition \ref{relativistic Carleman}):
\begin{multline}\notag 
Q^+(f,f)
=\frac{1}{2\pZ } \int_{\mathbb{R}^3}\frac{dp'}{\pZp }f(p')\int_{(p'^\mu-p^\mu)(q'_\mu-p_\mu)=0}\frac{d\pi_{q'}}{\qZp }\frac{s\sigma u(\pZp +\qZp -\pZ ) f(q')}{\left|(\pZp -\pZ )\frac{q}{\qZ }-(p'-p)\right|},
\end{multline}
where $u(x)$ is defined in \eqref{u}.
This expression is achieved by raising the 3-dimensional integral in \eqref{Q original}, on the mass-shell boundary to the four-dimensional integral, carrying out the Dirac-delta integration and applying the simple layer formula.
Then, a careful analysis of an intermediate form of the relativistic Carleman representation reveals that we need to establish a potential type estimate of $Q^+$
$$
\int_\rth \frac{(\pZ )^{1/2}}{g(p^\mu,a^\mu)}Q^+(f,f)dp,
$$
for any energy-momentum vector $a^\mu$,
and we need to establish the estimate of the integral of $Q^+$ restricted to  relativistic hyper-surfaces
$$
\int_\rth dp \hspace{1mm}(\pZ )^{1/2}Q^+(f,f) \delta(a^\mu(p_\mu-b_\mu)),
$$
for an arbitrary space-like 4-vector $a^\mu$ and  energy-momentum 4-vector $b^\mu$.
In general, we write the 
\textit{weak formulation} for the homogeneous relativistic Boltzmann equation \eqref{recall} as 
\begin{equation}\label{weakformulation}\partial_t \int_{\mathbb{R}^3} f\varphi(p)dp+\int_{\mathbb{R}^3}f Lf\varphi(p)dp=\int_{\mathbb{R}^3} Q^+(f,f)\varphi(p)dp. \end{equation} 
We then say that $f$ is a weak solution to the relativistic Boltzmann equation if it satisfies \eqref{weakformulation} for every $\varphi$ that is a nonnegative Radon measure. 

The key factor common for both estimates is to transform the integral by applying a suitable change of variables to the Dirac-delta representation of $Q$ from (\ref{Q original}) using a specific Lorentz transformation matrix given in \eqref{Lorentz},  which enables one
to work in the center-of-momentum frame.  Unlike most of the previous results where the specific form of the Lorentz transformation is irrelevant, however, we estimate the contribution of each row of the Lorentz transformation separately and show that only the first row matters in the estimate, which enables one to avoid estimating the highly prohibitive singularities in all the other rows of the type $1/|p\times q|$.
This all leads to the following control from below of the relative momentum:
\begin{align*}
\displaystyle\sqrt{-\left(\frac{\sqrt{s}}{2}-\Lambda^0_{~\mu} a^\mu\right)^2+\left|\frac{g}{2}\omega -\Lambda a\right|^2}
\gtrsim
\frac{\left|\frac{g}{2}\omega-\Lambda a\right|}{\sqrt{\aZ}\big(\pZ \qZ \big)^{1/2}}.
\end{align*}
This is shown in Lemma \ref{key lemma}.  The manipulations used to compute each row of the Lorentz transformation  separately, to the author’s best knowledge, have never previously been employed in the study of relativistic kinetic equations.

\subsection{A brief history of previous results}
In this section we will give a brief history of previous results in relativistic kinetic theory.  We will only emphasize the results that are most closely related to this paper.

The first global-in-time existence result for the relativistic Boltzmann equation  was obtained by Dudy\'{n}ski and Ekiel-Je\.{z}ewska for the linearized equation in \cite{D-E3,D} in 1988-89. The full nonlinear case was then  studied by Glassey and Strauss \cite{GS3,GS4} in 1993 and 1995.
The global existence, uniqueness and stability of the relativistic Vlasov-Maxwell-Landau system with self consistent electro-magnetic field was proven by Guo and Strain in \cite{Guo-Strain3} for nearby relativistic Maxwellian equilibrium initial data in 2004.
They further proved the global existence for the relativistic Vlasov-Maxwell-Boltzmann equation \cite{Guo-Strain2} near Maxwellian in 2012.  Here the center of momentum coordinates  and another set of coordinates were used in a complementary manner to control the singularities created by the derivatives of the post-collisional momentum variables in the relativistic collision operator.  For a systematic derivation of the center of momentum representation of the relativistic collision operator, see \cite{MR2765751}.
Glassey established in 2006 in \cite{GL-Vacuum} a relativistic counterpart of the the near-vacuum regime theory;  see also \cite{MR2679588}. For the Newtonian limit of the relativistic Boltzmann equation, see \cite{Cal,MR2679588}.  A study of blow-up  for the relativistic Boltzmann equation without the loss term can be found in \cite{MR2102321}.  We refer to \cite{D-E0,D-E0er,D-E2,MR2378164} for the Cauchy problem in the framework of the renormalized solutions \cite{D-L1,D-L2}.
In regards to the regularizing effect of the relativistic collision operator we have \cite{MR1402446, MR3880739, MR1450762}.

Works on various relaxation time approximations of the relativistic Boltzmann equation started recently.
See \cite{MR2988960,MR3300786,1801.08382} for the study of the Marle type relativistic BGK model,
and \cite{1811.10023} for the Anderson-Witting type relativistic BGK model.
Recently, a novel BGK type model was introduced in \cite{Pennisi_2018} and the existence is derived in \cite{H-R-Yun}.

 The $L^\infty$ propagations for the classical homogeneous Boltzmann equation is well established. We would like to mention the work of Bobylev \cite{MR1478067} in 1997, which shows that the exponentially-weighted $L^1$ estimates propagate in the 3-dimensional hard-sphere case. Motivated by this work, Gamba, Panferov and
Villani \cite{MR2533928} in 2009 proved that the exponentially-weighted $L^\infty$ (pointwise) estimates propagate by means of the Carleman
representation  for example under the assumption that the angular transition $b(\cos\theta)\sin^\alpha\theta$ is bounded for some $0 \le \alpha < d-1$. More recently, Alonso, Gamba, and Taskovic \cite{1711.06596} extended this result under the more general assumption in the angular transition that $b(\theta)\in L^1(\mathbb{S}^2)$ and improved the decay rate to the Maxwellian equilibrium. 

In contrast to the  Newtonian case where the homogeneous theory for the Boltzmann equation is well established,
the literature on the spatially homogeneous relativistic Boltzmann equation is very limited.
The Cauchy problem for various cosmological models is studied in \cite{MR3169776}. The existence and various moment estimates are studied in \cite{MR3166961}.  In \cite{StrainTas}, the entropy dissipation estimate was shown for weak solutions to the spatially homogeneous  relativistic Landau equation.   Then, with that estimate, the global existence of a standard weak solution was established as well as the propagation of any high order polynomial moment.  In \cite{1903.05301} the conditional uniqueness of a weak solution was shown for the spatially homogeneous relativistic Landau equation.  Further general references on relativistic and non-relativistic kinetic equations can be found in \cite{C,C-I-P,C-K,E-M-V,DeGroot,GL1996,Vil02}.

\subsection{Outline of the remainder of this article}
The rest of this paper is organized as follows. In Section \ref{prelim}, we present various useful technical lemmas. In Section \ref{egain},
we establish a potential type estimate of $Q^+$ and an estimate of $Q^+$ restricted to relativistic hyper-surfaces. In Section \ref{esol}, we
use the $Q^+$ estimates of Section \ref{egain} to obtain corresponding uniform estimates for the solutions to \eqref{RBE}. We then derive the relativistic Carleman representation
in Section \ref{pupper}, and then we use it to establish a uniform bound on $Q^+$ for solutions to \eqref{RBE}.  This uniform bound for $Q^+$ then yields Theorem \ref{main}. Section \ref{pmaxwell} is devoted to the proof of the
propagation of uniform polynomial and Maxwellian upper bounds. In Sections \ref{htheorem} and \ref{asympt}, we present two applications of our main results, namely, the celebrated $H$-theorem and the asymptotic behavior of solutions respectively.

%
%
%
%
%

\section{Preliminary estimates}\label{prelim}
In this section we will introduce some technical lemmas that will be crucially used later in the paper. We start with the following well known coercive inequality for the relative momentum in the center of momentum framework.

\begin{lemma}[Lemma 3.1 (i) on page 316 of \cite{GS3}]\label{coersive inequality}  The relative momentum $g$ satisfies the following inequalities:
\begin{equation}\label{gINEQ}
\frac{|p-q|}{\sqrt{\pZ \qZ }}\leq g(p^{\mu},q^{\mu})\leq |p-q|.
\end{equation}
\end{lemma}
\begin{proof}
It is very easy to see that 
$$g=\sqrt{-(p^0-q^0)^2+|p-q|^2}\le |p-q|.$$ For the lower-bound, we observe from \eqref{g} that
\begin{multline*}
g^2=2(p^0q^0-p\cdot q -1)
=\frac{2 ( (p^0q^0)^2-(p\cdot q +1)^2)}{p^0q^0+p\cdot q +1}\\=\frac{2 ( (1+|p|^2)(1+|q|^2)-(p\cdot q +1)^2)}{p^0q^0+p\cdot q +1}\\
\ge \frac{2 (|p|^2+|q|^2+|p|^2|q|^2-(p\cdot q)^2-2p\cdot q)}{2p^0q^0}
=\frac{|p-q|^2+|p\times q|^2}{p^0q^0}.
\end{multline*}
Above we use the Cauchy-Schwarz inequality as $p^0q^0\ge p\cdot q +1.$
\end{proof} 
In the next lemma, we derive a uniform lower bound estimate for the loss term.

\begin{lemma}[Lemma 3.3 on page 925 of \cite{MR3166961}]\label{lowerL}
	Let $f(p,t)$ have finite mass, energy, and entropy as in \eqref{conlaw}, \eqref{entropy.eq}, \eqref{initial} and \eqref{Sol}. Then there exists uniform positive constants $C_{\ell}>0$ and $C_u>0$, which are determined only by the mass, energy, and entropy of the initial data $f_0$, such that the following estimate holds:
$$
C_{\ell} \pZ \leq
\int_{\rth\times \mathbb{S}^2}\mollrV  \sigma(g,\theta)f(q,t)d\omega dq
= (Lf)(p,t)
\leq C_u \pZ.
$$
\end{lemma}
We remark that this lemma holds for a more general kernel such as $\sigma(g,\theta)=g^{\rho}\sigma_0(\theta)$ as long as $0\le \rho\le 1$ and $\int_{\mathbb{S}^2}\sigma_0(\theta)d\omega$ is constant, see \cite{MR3166961}.

The two lemmas above will be used in the mathematical developments below.  We will now discuss a few elementary aspects of {\it Lorentz transformations} which will also be useful throughout the rest of this paper.
Let $\Lambda$ be a $4\times 4$ matrix (of real numbers) denoted by
$
\Lambda=(\Lambda_{\ \nu}^{\mu})_{0\le \mu,\nu\le 3}.
$
The matrix $\Lambda$ is called a (proper) Lorentz transformation if $\mbox{det}(\Lambda) = 1$ and
\begin{equation}\label{lorentz.trans.inv}
\Lambda_{\ \mu}^{\kappa}  \eta_{\kappa \lambda}  \Lambda_{\ \nu}^{\lambda} = \eta_{\mu \nu}, \qquad (\mu, \nu = 0,1,2,3).
\end{equation}
This implies the following invariance of the Lorentz inner product from \eqref{lorentz.inner.prd}:
$$
p^\kappa q_\kappa = p^\kappa \eta_{\kappa\lambda } q^{\lambda}
=
(\Lambda^{\kappa}_{~\mu}  p^{\mu})\eta_{\kappa\lambda } (\Lambda^{\lambda}_{~\nu}  q^{\nu} ).
$$
For a Lorentz transformation with components $\Lambda_{\ \nu}^{\mu}$, then $\mu$ denotes the column and $\nu$ denotes the row as in \eqref{Lorentz} below.
Then any such $\Lambda$ is invertible and
the inverse matrix is denoted $\Lambda^{-1} = (\Lambda_{\mu}^{~\nu})_{0\le \mu,\nu\le \dim}$.  Further the inverse  $(\Lambda^{-1})^{\mu}_{~\nu}  = \Lambda^{~\mu}_{\nu}$
is also a Lorentz transformation.  We refer to \cite{MR2707256,St,MR2679588,MR2793935} and the references therein for further discussions of Lorentz transformations.

We now define the following specific Lorentz transform $\Lambda$ which will be used throughout the paper:
\begin{equation}\label{Lorentz}
\Lambda=(\Lambda^{\mu}_{~\nu})
=
\left(\begin{array}{cccc}\frac{\pZ +\qZ }{\sqrt{s}}& -\frac{p^1+q^1}{\sqrt{s}} &-\frac{p^2+q^2}{\sqrt{s}} &-\frac{p^3+q^3}{\sqrt{s}}\\ \Lambda^{1}_{~0} &\Lambda^{1}_{~1}&\Lambda^{1}_{~2}&\Lambda^{1}_{~3}\\ 0 & \frac{(p\times q)^1}{|p\times q|} &\frac{(p\times q)^2}{|p\times q|}&\frac{(p\times q)^3}{|p\times q|}\\ \frac{\pZ -\qZ }{g} &-\frac{p^1-q^1}{g}&-\frac{p^2-q^2}{g}&-\frac{p^3-q^3}{g}\end{array}\right),
\end{equation}
where the second row is given by
$$
\Lambda^{1}_{~0}=\frac{2|p\times q|}{g\sqrt{s}},
\quad \Lambda^{1}_{~i}=\frac{2\left(p^i \{\pZ +\qZ  p^\mu q_\mu\}+q_i\{\qZ +\pZ p^\mu q_\mu\}\right)}{g\sqrt{s}|p\times q|}
\quad (i=1,2,3).
$$
We note that this matrix satisfies \eqref{lorentz.trans.inv}.  The matrix $\Lambda$ also satisfies the following identities for energy-momentum vectors $p^\mu$ and $q^\mu$  (for $\mu = 0,1,2,3$):
 \begin{align}\label{center.mom.lorentz}
\Lambda^{\mu}_{~\nu}(p^\nu+q^\nu) =(\sqrt{s}, 0, 0, 0),~ \mbox{ and }~ -\Lambda^{\mu}_{~\nu}(p^\nu-q^\nu) = (0, 0,0,g)
\end{align}
where $s\eqdef s(p^\mu, q^\mu)$ is given by \eqref{s} and $g$ is given by \eqref{g}.  The specific form of this Lorentz matrix was given in \cite{MR2707256,St,MR2679588} where these details were explained.

\begin{lemma}\label{Lorentz T}
Every element in the first row of (\ref{Lorentz}) satisfies
$$
|\Lambda^{0}_{~\nu} |
\lesssim
(\pZ \qZ )^{1/2},
\quad
\nu = 0, 1, 2, 3.
$$
\end{lemma}

\begin{proof} From  \eqref{s}  we have
$
s=(\pZ +\qZ )^2-|p+q|^2
$
which allows us to compute
\begin{align*}
|\Lambda^{0}_{~0}|&=\frac{\pZ +\qZ }{\sqrt{(\pZ +\qZ )^2-|p+q|^2}}\cr
&=\frac{\pZ +\qZ }{\sqrt{(\pZ +\qZ -|p+q|)(\pZ +\qZ +|p+q|)}}\cr
&\leq \frac{(\pZ +\qZ )^{1/2}}{\sqrt{\pZ +\qZ -|p+q|}},
\end{align*}as $|p+q|\geq 0.$
Then we observe that
$$
\pZ +\qZ -|p+q|\geq \pZ +\qZ -|p|-|q|\geq \frac{1}{2\pZ }+\frac{1}{2\qZ }=\frac{\pZ +\qZ }{2\pZ \qZ },
$$
since
$$
\pZ -|p|=\frac{(\pZ)^2-|p|^2}{\pZ +|p|}=\frac{1}{\pZ +|p|}\geq \frac{1}{2\pZ }.
$$
Thus we obtain
$$
\Lambda^{0}_{~0}\leq \sqrt{2}(\pZ \qZ )^{1/2}.
$$
The remaining part can be treated similarly, for $j=1,2$, or $3$ we have
$$
| \Lambda^{0}_{~j}|\leq\frac{|p+q|}{\sqrt{s}}\leq \frac{\pZ +\qZ }{\sqrt{s}}=|\Lambda^{00}|\leq  \sqrt{2}(\pZ \qZ )^{1/2}.
$$
This completes the proof.
\end{proof}

This completes our discussion of the preliminary estimates.  In the next section we will prove estimates for $Q^+$ from \eqref{Q original} for an arbitrary function $f\ge 0$.

\section{Estimates of the gain term}\label{egain}
In this section, we provide two necessary estimates for the gain term of the collision operator.  The first estimate in Lemma \ref{key lemma} is a pointwise
 potential type estimate.  The second estimate in Proposition \ref{lemma1} studies the integration of the gain term over relativistic hyper-surfaces.
Note that throughout this section we assume that the arbitrary non-negative function $f$ is not necessarily a solution to (\ref{RBE}).

\subsection{Potential type estimate of $Q^+$.}
First we  obtain an estimate which gives control on the relative momentum.
Notice that the estimates for the rows of $\Lambda$ other than the first row are systematically avoided, and the singularities in those rows therefore do not result in any harm.

\begin{lemma}\label{key lemma} Let $\Lambda$ be the Lorentz transform in \eqref{Lorentz}, and let $a^{\mu}$ be an arbitrary energy-momentum vector. Then we have
\begin{align*}
\displaystyle\sqrt{-\left(\frac{\sqrt{s}}{2}-\Lambda^0_{~\mu} a^\mu\right)^2+\left|\frac{g}{2}\omega -\Lambda a\right|^2}
\gtrsim
\frac{\left|\frac{g}{2}\omega-\Lambda a\right|}{\sqrt{\aZ}\big(\pZ \qZ \big)^{1/2}}.
\end{align*}
In the equation above and in the proof below we slightly  abused notation to define $\Lambda a \eqdef ((\Lambda a)_1,(\Lambda a)_2,(\Lambda a)_3)$ by $(\Lambda a)_i\eqdef \Lambda^i_{~\mu} a^\mu$ for $i=1,2,3$.
\end{lemma}

\begin{proof}
We now define a 4-vector $A^\mu$ by
\[
A^\mu = \left(\frac{\sqrt{s}}{2},\frac{g}{2}\omega\right).
\]
Then we observe that both $A^\mu$ and $\Lambda^\mu_{~\nu}  a^\nu$ are energy-momentum vectors as
$$
A^\mu A_\mu= -\frac{s}{4}+\frac{g^2}{4}=-1,
$$
and similarly for $\Lambda^\mu_{~\nu}  a^\nu$  using \eqref{lorentz.trans.inv}.  Therefore, we see from (\ref{g}) that
$$
\sqrt{-\left(\frac{\sqrt{s}}{2}-\Lambda^0_{~\nu}  a^\nu\right)^2+\left|\frac{g}{2}\omega -\Lambda a\right|^2}
=
g(A^\mu, \Lambda^\mu_{~\nu}  a^\nu ).
$$
Thanks to the coercive inequality in Lemma \ref{coersive inequality}, we derive
\begin{equation}\label{coera}
\sqrt{-\left(\frac{\sqrt{s}}{2}-\Lambda^0_{~\nu}  a^\nu \right)^2+\left|\frac{g}{2}\omega -\Lambda a\right|^2}
\geq
\frac{\left|\frac{g}{2}\omega-\Lambda a\right|}{\sqrt{\frac{\sqrt{s}}{2}\Lambda^0_{~\nu}  a^\nu}}.
\end{equation}
Then, we apply Lemma \ref{Lorentz T} and $s\lesssim \pZ \qZ $ to get the desired result:
$$
\sqrt{\frac{\sqrt{s}}{2}\Lambda^0_{~\nu}  a^\nu}
\lesssim
(\pZ \qZ )^{\frac{1}{4}+\frac{1}{4}}\sqrt{\aZ}
\lesssim
(\pZ \qZ )^{1/2}\sqrt{\aZ},
$$
which holds using $1+\sum_{j=1}^3 (a^j)^2 = (\aZ)^2$.
This completes the proof.
\end{proof}

We now prove a potential type estimate for the relativistic collision operator:

\begin{proposition}
Suppose $f\in L^1_1(\mathbb{R}^3)$. Then,
 for an arbitrary energy-momentum vector $a^\mu$ we have
\label{lemma1}
$$
\int_\rth \frac{(\pZ )^{1/2}}{g(p^\mu,a^\mu)}Q^+(f,f)dp\lesssim \sqrt{\aZ}\|f\|^2_{L^1_{1}},
$$
where we recall the definition \eqref{g}
\end{proposition}

\begin{proof}By choosing $\varphi(p)=\frac{(p^0)^{1/2}}{g(p^\mu,a^\mu)}$ in \eqref{weakformulation}, we observe that the right-hand side of the weak formulation is equal to
\begin{multline*}
I\eqdef \int_\rth dp \frac{(\pZ )^{1/2}}{g(p^\mu,a^\mu)}Q^+(f,f)
\\
= \frac{1}{2}\int_{\rth}\frac{dp}{\pZ }\frac{(\pZ )^{1/2}}{g(p^\mu,a^\mu)}\int_{\mathbb{R}^3}\frac{dq}{{\qZ }}\int_{\mathbb{R}^3}\frac{dq'\hspace{1mm}}{{\qZp }}\int_{\mathbb{R}^3}\frac{dp'\hspace{1mm}}{{\pZp }}  sg \delta^{(4)}(p^\mu +q^\mu -p'^\mu -q'^\mu)f(p')f(q').
\end{multline*}
Doing a pre-post relabelling of the variables $(p^{\mu}, q^{\mu}) \to (p^{\prime\mu}, q^{\prime\mu})$ and using the fact that  $s$ and $g$ are invariant under this transformation, we obtain that the integral $I$ is equal to
$$
\frac{1}{2} \int_{\rth}\frac{dp}{\pZ }\int_{\mathbb{R}^3}\frac{dq}{{\qZ }}\int_{\mathbb{R}^3}\frac{dq'\hspace{1mm}}{{\qZp }}\int_{\mathbb{R}^3}\frac{dp'\hspace{1mm}}{{\pZp }}\hspace{1mm}  \frac{(\pZp )^{1/2} sg}{g(p'^\mu,a^\mu)}
  \delta^{(4)}(p'^\mu +q'^\mu -p^\mu -q^\mu)f(p)f(q).
$$
We then use  $\pZp \leq \pZ +\qZ$ from \eqref{collision.invariants}, to see that $I$ is bounded above by
$$
\int_{\rth}\frac{dp}{\pZ }\int_{\mathbb{R}^3}\frac{dq}{{\qZ }}\int_{\mathbb{R}^3}\frac{dq'\hspace{1mm}}{{\qZp }}\int_{\mathbb{R}^3}\frac{dp'\hspace{1mm}}{{\pZp }}
  \frac{(\pZ +\qZ )^{1/2}sg}{g(p'^\mu,a^\mu)}
  \delta^{(4)}(p'^\mu +q'^\mu -p^\mu -q^\mu)f(p)f(q).
$$
The next estimate will be performed in the \textit{center-of-momentum} frame where $p + q = 0$.
For this, we make a change of variable using the specific choice of the Lorentz transform given in \eqref{Lorentz} as follows:
\begin{equation}\label{lambda.change.lorentz}
	\Lambda^{\mu}_{~\nu}  p'^{\nu}\eqdef P^{\prime\mu}=(P^{\prime 0} ,P'), \quad
	\Lambda^{\mu}_{~\nu} q'^{\nu}\eqdef Q^{\prime\mu}=(Q^{\prime 0} ,Q').
\end{equation}
Then, we will use the Lorentz invariance of $\delta^{(4)}$ as follows
$$
\delta^{(4)}(p'^\mu +q'^\mu -p^\mu -q^\mu)=\delta^{(4)}(\Lambda^{\mu}_{~\nu} (p'^{\nu} +q'^{\nu} -p^{\nu} -q^{\nu})),
$$
and we will similarly use the Lorentz invariance of $g$ as
\[
g(p^{\prime\mu},a^{\mu})=g(\Lambda^{\mu}_{~\nu} p^{\prime\nu},\Lambda^{\mu}_{~\nu} a^{\nu})=g(P^{\prime \mu},\Lambda^{\mu}_{~\nu} a^{\nu}).
\]
We also remark that $\frac{dp}{\pZ }$ is a Lorentz transformation invariant measure as in \eqref{lorentz.invariant.m}.
Now we can bound the integral $I$ from above using this change of variable as
\begin{align*}
I&\lesssim
\int_{\rth}\frac{dp}{\pZ }\int_{\mathbb{R}^3}\frac{dq}{{\qZ }}\int_{\mathbb{R}^3}\frac{dQ'\hspace{1mm}}{{Q'^0 }}
\int_{\mathbb{R}^3}\frac{dP'}{P'^0 }\frac{(\pZ +\qZ )^{1/2} sg}{g(P'^\mu,\Lambda^{\mu}_{~\nu} a^{\nu})}\\
	&\qquad\times\delta^{(4)}(P'^\mu +Q'^\mu -\Lambda^{\mu}_{~\nu} (p^\nu +q^\nu))f( p)f( q)\\
&\lesssim
\int_{\rth}\frac{dp}{\pZ }\int_{\mathbb{R}^3}\frac{dq}{{\qZ }}sg(\pZ +\qZ )^{1/2}f( p)f( q)\\
&\qquad\times\int_{\mathbb{R}^3}\frac{dQ'\hspace{1mm}}{{Q'^0 }}\int_{\mathbb{R}^3}\frac{dP'}{P'^0}\frac{1}{g(P'^\mu,\Lambda^{\mu}_{~\nu} a^{\nu})}
	\delta(P'^0 +Q'^0 -\sqrt{s})\delta^{(3)}(P'+Q'),
\end{align*}
where we used \eqref{center.mom.lorentz}.  Therefore, carrying out the integration over $Q'$, we obtain
\begin{align*}
I&\lesssim \int_{\rth}\frac{dp}{\pZ }\int_{\mathbb{R}^3}\frac{dq}{{\qZ }}sgf( p)f( q)(\pZ +\qZ )^{1/2}\int_{\mathbb{R}^3}\frac{dP'}{P'^0}\frac{1}{g(P'^\mu,\Lambda^{\mu}_{~\nu} a^{\nu})}\frac{\delta(2P'^0 -\sqrt{s})}{P'^0 }.
\end{align*}
We now introduce a step function
\begin{align}\label{u}
u(x)=
\begin{cases}
~1& x>0,\cr
~0& x\leq0,
\end{cases}
\end{align}
and raise the 3-dimensional integral with respect to $P'$ to a 4-dimensional integral with respect to $P'^\mu$ as follows:
\begin{align*}
\int_{\rth}\frac{dp}{\pZ }\int_{\mathbb{R}^3}\frac{dq}{{\qZ }}sgf( p)f( q)(\pZ +\qZ )^{1/2}
\int_{\mathbb{R}^4}dP'^\mu\frac{u(P'^0 )\delta(P'^\mu P'_\mu+1)}{g(P'^\mu,\Lambda^{\mu}_{~\nu} a^{\nu})}
\frac{\delta(P'^0  -\frac{\sqrt{s}}{2})}{P'^0 },
\end{align*}
Above we also used the Lorentz invariant property of the measure $\frac{dP'}{P'^0 }$ as
\begin{equation}\label{lorentz.invariant.m}
\int_{\rth}\frac{dP'}{P'^0 }=\int_{\rfo}dP'^\mu \ u(P'^0 )\delta(P'^\mu P'_\mu+1).
\end{equation}
We will also use the following calculation, recalling $s=g^2+4$, to obtain
\begin{align*}
\delta\left(\frac{\sqrt{s}}{2}-P'^{0} \right)&=\delta\left(\frac{s/4-1-|P'|^2}{\sqrt{s}/2+P'^{0} }\right)\cr
&=\delta\left(\frac{g^2/4-|P'|^2}{\sqrt{s}}\right)\cr
&=\delta\left(\frac{(g/2-|P'|)(g/2+|P'|)}{\sqrt{s}}\right)\cr
&=\frac{\sqrt{s}}{g}\delta\left(g/2-|P'|\right).
\end{align*}
To get to the last line above we also used that
\[
\delta\left(|P'|^2-\frac{g^2}{4}\right)= \frac{\delta(|P'|-\frac{g}{2})+\delta(|P'|+\frac{g}{2})}{g},
\]
and the fact that $\delta(|P'|+\frac{g}{2})$ causes that integral to be zero.

Then, carrying out $\delta(P'^0  -\frac{\sqrt{s}}{2})$ and using \eqref{lorentz.inner.prd}, we have
\begin{align*}
I
&\lesssim
\int_{\rth}\frac{dp}{\pZ }\int_{\mathbb{R}^3}\frac{dq}{{\qZ }}sgf( p)f( q)(\pZ +\qZ )^{1/2}\int_{\mathbb{R}^3}dP'\frac{\delta(|P'|^2-\frac{g^2}{4})}{g(P'^\mu,\Lambda^{\mu}_{~\nu} a^{\nu})}
	\frac{2}{\sqrt{s}}\\
&\lesssim
\int_{\rth}\frac{dp}{\pZ }\int_{\mathbb{R}^3}\frac{dq}{{\qZ }}\sqrt{s}f( p)f( q)(\pZ +\qZ )^{1/2}\int_{\mathbb{R}^3}dP'\frac{1}{g(P'^\mu,\Lambda^{\mu}_{~\nu} a^{\nu})}
	\delta\left(|P'|-\frac{g}{2}\right).
\end{align*}
Now writing $P'=|P'|\omega$ in polar coordinates, we have
\begin{align*}
I&
\lesssim
\int_{\rth}\frac{dp}{\pZ }\int_{\mathbb{R}^3}\frac{dq}{{\qZ }}\sqrt{s}f( p)f( q)(\pZ +\qZ )^{1/2}\int_{\mathbb{S}^2}d\omega\\
&\qquad\times \int_0^\infty d|P'|\ |P'|^2\delta\left(|P'|-\frac{g}{2}\right)\frac{1}{\sqrt{-\left(\frac{\sqrt{s}}{2}-\Lambda^{0}_{~\nu} a^{\nu}\right)^2+\left||P'|\omega -\Lambda a\right|^2}},
\end{align*}
where we denote the vector $\Lambda^{\mu}_{~\nu} a^{\nu} = (\Lambda^{0}_{~\nu} a^{\nu}, \Lambda a).$ Computing the delta function:
\begin{equation}\label{beforedw}
 I
 \lesssim
 \int_{\rth}\frac{dp}{\pZ }\int_{\mathbb{R}^3}\frac{dq}{{\qZ }}\sqrt{s}f( p)f( q)\int_{\mathbb{S}^2}d\omega \frac{g^2}{4}\frac{(\pZ +\qZ )^{1/2}}{\sqrt{-\left(\frac{\sqrt{s}}{2}-\Lambda^{0}_{~\nu} a^{\nu}\right)^2+\left|\frac{g}{2}\omega -\Lambda a\right|^2}}.
\end{equation}
We then use Lemma \ref{key lemma} to bound $I$ by
$$
 I
 \lesssim
 \int_{\rth}\frac{dp}{\pZ }\int_{\mathbb{R}^3}\frac{dq}{{\qZ }}g\sqrt{s}f( p)f( q)\int_{\mathbb{S}^2}d\omega \frac{\pZ \qZ \sqrt{\aZ}}{\left|\omega-\frac{2\Lambda a}{g}\right|}.
$$
Using $\int_{\mathbb{S}^2}d\omega \frac{1}{\left|\omega-z\right|}\lesssim 1$ independent of $z\in \rth$, we have
$$
I\lesssim \int_{\rth}\frac{dp}{\pZ }\int_{\mathbb{R}^3}\frac{dq}{{\qZ }}g\sqrt{s}f( p)f( q)\pZ \qZ \sqrt{\aZ}.
$$
Finally, we employ $g\leq \sqrt{s}\lesssim \sqrt{\pZ \qZ }$ to get the result in Proposition \ref{lemma1}.
\end{proof}

%
%
%
%

\subsection{Hyper-surface integral of $Q^+$} We now estimate the integral of the gain term $Q^+$ on a relativistic hypersurface.
\begin{proposition}\label{lemma2}
Let $a^\mu$ be a space-like vector, and $b^\mu$ be an energy-momentum vector.
Suppose $f\in L^1_{1}(\mathbb{R}^3)$. Then we have
\begin{align}\label{hyper-surface}
\int_\rth dp \hspace{1mm}(\pZ )^{1/2}Q^+(f,f) \delta(a^\mu(p_\mu-b_\mu))\lesssim  \frac{1}{\sqrt{a^{\mu}a_{\mu}}}\|f\|^2_{L^1_{\frac{1}{2}}}.
\end{align}
\end{proposition}

\begin{proof}By choosing $\varphi(p)=(p^0)^{1/2}\delta(a^\mu(p_\mu-b_\mu))$ in \eqref{weakformulation}, we observe that the right-hand side of the weak formulation is equal to
\begin{multline}\notag
I \eqdef
\int_{\rth}\frac{dp}{\pZ}\int_{\mathbb{R}^3}\frac{dq}{{\qZ}}\int_{\mathbb{R}^3}\frac{dq'\hspace{1mm}}{{\qZp}}\int_{\mathbb{R}^3}\frac{dp'\hspace{1mm}}{{\pZp}}\hspace{1mm}(\pZ)^{1/2}sg f(p')f(q')
\\
\times
\delta^{(4)}(p^\mu +q^\mu -p'^\mu -q'^\mu)\delta(a^\mu(p_\mu-b_\mu)).
\end{multline}
Similar to the previous proof, we again do a pre-post relabelling of the variables $(p^{\mu}, q^{\mu}) \to (p^{\prime\mu}, q^{\prime\mu})$ and use the fact that  $s$ and $g$ are invariant under this transformation, we obtain that the integral $I$ is bounded above by
\begin{align*}
I&\lesssim
\int_{\rth}\frac{dp}{\pZ }\int_{\mathbb{R}^3}\frac{dq}{{\qZ }}\int_{\mathbb{R}^3}\frac{dq'\hspace{1mm}}{{\qZp }}\int_{\mathbb{R}^3}\frac{dp'\hspace{1mm}}{{\pZp }} (\pZp )^{1/2}sg f(p)f(q)
\cr
&\qquad\times \delta^{(4)}(p^\mu +q^\mu -p'^\mu -q'^\mu) \delta(a^\mu(p'_\mu-b_\mu))\cr
&\lesssim
\int_{\rth}\frac{dp}{\pZ }\int_{\mathbb{R}^3}\frac{dq}{{\qZ }} (\pZ +\qZ )^{1/2}sg f(p)f(q) \int_{\mathbb{R}^3}\frac{dq'\hspace{1mm}}{{\qZp }}\int_{\mathbb{R}^3}\frac{dp'}{{\pZp }}\cr
&\qquad\times\delta^{(4)}(p^\mu +q^\mu -p'^\mu -q'^\mu)\delta(a^\mu(p'_\mu-b_\mu)),
\end{align*}
where we used $\pZp \leq \pZ +\qZ $.
Similar to the proof of Proposition \ref{lemma1}, we will use the change of variable \eqref{lambda.change.lorentz}.
We also define $A^\mu$ and $B^\mu$ by
\[
A^\mu=(A^0,A)=\Lambda^{\mu}_{~\nu} a^\nu,\quad B^\mu=(B^0,B)=\Lambda^{\mu}_{~\nu} b^\nu.
\]
Then, following the same argument as used in the proof of Proposition \ref{lemma1}, the 3-dimensional delta function of the momentum conservation laws in the 4-dimensional delta function reduces the $dQ'$ integral and then we can bound $I$ above by
\begin{multline}\label{turn back}
\int_{\rth}\frac{dp}{\pZ }\int_{\mathbb{R}^3}\frac{dq}{{\qZ }}sgf(p)f(q)(\pZ +\qZ )^{1/2}\int_{\mathbb{R}^3}\frac{dP'}{P'^{0} } \delta\left(\frac{\sqrt{s}}{2}-P'^{0} \right)\frac{\delta\left(A^\mu(P'_\mu-B_\mu)\right)}{P'^{0}}
\\
=
\int_{\rth}\frac{dp}{\pZ }\int_{\mathbb{R}^3}\frac{dq}{{\qZ }}sgf(p)f(q)(\pZ +\qZ )^{1/2}
\int_{\mathbb{R}^3}\frac{dP'\hspace{1mm}}{(P'^{0} )^2} \frac{\sqrt{s}}{g}\delta\left(\frac{g}{2}-|P'|\right)\delta\left(A^\mu(P'_\mu-B_\mu)\right)
\\
=
\int_{\rth}\frac{dp}{\pZ }\int_{\mathbb{R}^3}\frac{dq}{{\qZ }}sgf(p)f(q)(\pZ +\qZ )^{1/2}\mathcal{B}(p,q),
\end{multline}
up to a constant.
We now express $\mathcal{B}$ in the polar coordinates, with $\theta$ denoting the angle between $A$ and $P'$, as
\begin{align*}
\mathcal{B}
&=
\int_{\mathbb{R}^3}\frac{dP'}{(P'^{0} )^2} \frac{\sqrt{s}}{g}\delta\left(\frac{g}{2}-|P'^{0} |\right)
\delta\left(-\aZ P'^{0} +A\cdot P'-A^{\mu}B_{\mu}\right)\cr
&=
\int_{\mathbb{S}^2}d\omega\int_0^{\infty}\frac{|P'|^2d|P'|}{(P'^{0} )^2} \frac{\sqrt{s}}{g}\delta\left(\frac{g}{2}-|P'|\right)
\delta\left(-\aZ\sqrt{1+|P'|^2}+|A||P'|\cos\theta-A^{\mu}B_{\mu}\right)\cr
&=
\frac{g}{\sqrt{s}}\int_{\mathbb{S}^2}d\omega~
\delta\left(-\frac{\sqrt{s}}{2}\aZ+|A|\frac{g}{2}\cos\theta-A^{\mu}B_{\mu}\right)\cr
&=
\frac{2}{\sqrt{s}|A|}
\int_{\mathbb{S}^2}d\omega ~ \delta\left(\cos\theta-\frac{2A^\mu B_\mu+\aZ\sqrt{s}}{|A|g}\right).
\end{align*}
Note that we have used $|P'|=\frac{g}{2}$ and $P'^{0} =\frac{\sqrt{s}}{2}\geq 1$
when we carried out the delta function of $\delta(g/2-|P'|)$ as also done in the proof of Proposition \ref{lemma1}.

We then make a change of variable $v=\cos\theta$ to compute that
\begin{align*}
\mathcal{B}=\frac{2}{\sqrt{s}|A|}\int_{0}^{2\pi}d\psi\int_{-1}^{1}dv ~ \delta\left(v-\frac{2A^\mu B_\mu+\aZ\sqrt{s}}{|A|g}\right)\leq \frac{4\pi}{\sqrt{s}|A|}.
\end{align*}
We put this estimate back into (\ref{turn back}) to obtain the following upper bound
\begin{align*}
I
&\lesssim \frac{1}{|A|}\int_{\rth}\frac{dp}{\pZ }\int_{\mathbb{R}^3}\frac{dq}{{\qZ }}\sqrt{s}gf(p)f(q)(\pZ +\qZ )^{1/2}\cr
&\lesssim \int_{\rth}dp\int_{\mathbb{R}^3}dq\ (\pZ \qZ )^{\frac{1}{2}}f(p)f(q)\frac{1}{|A|} \lesssim \frac{1}{|A|}\|f\|^2_{L^1_\frac{1}{2}},
\end{align*}
where we also used that $g<\sqrt{s}\lesssim \sqrt{\pZ \qZ}$.  Now, the desired result follows from
$$
|A|^2\geq |A|^2-(A^0)^2=\Lambda^{\mu}_{~\nu} a^{\nu} \eta_{\mu \lambda}\Lambda^{\lambda}_{~\kappa} a^{\kappa}=a^\mu a_\mu=-(\aZ)^2+|a|^2>0,
$$
which holds since $a^\mu$ is space-like.
This completes the proof.
\end{proof}

\section{Estimates of the solutions}\label{esol}
In this section we will establish corresponding potential type estimates and hyper-surface integral estimates of any solution to \eqref{RBE}.  Specifically, in contrast to the results in the previous section, we now in this section assume that $f(p,t)$ is a solution to \eqref{RBE} and the proofs below will make use of the dynamics of the Boltzmann equation \eqref{RBE}.

\begin{lemma}\label{lemma3}
Suppose $f_0$ satisfies (\ref{initial}) with $\exm\ge 1$. Let $f$ be a solution to (\ref{RBE}) satisfying (\ref{Sol}).  Then we have
$$
\int_\rth dp\hspace{1mm}(\pZ )^{1/2} \frac{f(p,t)}{g(p^\mu,a^\mu)}\lesssim (\aZ)^{1/2}\left(\|f_0\|_{L^\infty_{1}}+\|f_0\|_{L^1_1}\right)
+\frac{\sqrt{\aZ}}{C_{\ell}}\|f_0\|_{L^1_1}^2(1-e^{-C_{\ell}t}),
$$
for any energy-momentum vector $a^\mu$ with $g(p^\mu,a^\mu)$ defined in \eqref{g}.   Here $C_{\ell}>0$ is the constant from Lemma \ref{lowerL}.
\end{lemma}

We remark that the proof of Lemma \ref{lemma3} below only uses the entropy bounds from \eqref{initial} and \eqref{Sol} in the application of the lower bound for the \textit{collision frequency} from Lemma \ref{lowerL}.

\begin{proof}
	We multiply (\ref{recall}) by $\frac{(\pZ )^{1/2}}{g(p^\mu,a^\mu)}$ and integrate with respect to $dp$:
$$
\partial_t\int_{\rth} dp\hspace{1mm} (\pZ )^{1/2}\frac{f(p)}{g(p^\mu,a^\mu)}+\int_{\rth}dp \hspace{1mm} (\pZ )^{1/2}\frac{f(p)Lf(p)}{g(p^\mu,a^\mu)}\leq\int_{\rth} dp\hspace{1mm}(\pZ )^{1/2} \frac{Q^+(f,f)}{g(p^\mu,a^\mu)}.
$$
We then apply Lemma \ref{lowerL} and Proposition \ref{lemma1} and \eqref{conlaw} to derive
$$
\partial_t\int_{\rth} dp\hspace{1mm} (\pZ )^{1/2}\frac{f(p)}{g(p^\mu,a^\mu)}+C_{\ell}\int_{\rth}dp \hspace{1mm} (\pZ )^{1/2}\frac{f(p)}{g(p^\mu,a^\mu)}\lesssim
\sqrt{\aZ}\|f_0\|^2_{L^1_{1}},
$$
which leads to
	$$
	\frac{d}{dt}\left(e^{C_{\ell}t}\int_{\rth}dp\hspace{1mm}(\pZ )^{1/2} \frac{f(p,t)}{g(p^\mu,a^\mu)}\right)\lesssim \sqrt{\aZ} \|f_0\|^2_{L^1_{1}} e^{C_{\ell}t}.
	$$
Therefore, we conclude that
\begin{multline}\notag
\int_{\rth}dp ~(\pZ )^{1/2}\frac{f(p,t)}{g(p^\mu,a^\mu)}
\\
\lesssim
e^{-C_{\ell}t} \int_{\rth}dp\hspace{1mm}(\pZ )^{1/2} \frac{f_0(p)}{g(p^\mu,a^\mu)}
+
\frac{\sqrt{\aZ}}{C_{\ell}} \|f_0\|^2_{L^1_{1}} (1-e^{-C_{\ell}t}).
\end{multline}
We then use the coercive inequality in \eqref{gINEQ} to compute
\begin{align*}
&\int_{\rth}dp\hspace{1mm}(\pZ )^{1/2} \frac{f_0(p)}{g(p^\mu,a^\mu)}\cr
&\hspace{1cm}\leq \sqrt{\aZ}\int_{\rth}dp\hspace{1mm}\pZ  \frac{f_0(p)}{|p-a|}\\
&\hspace{1cm}= \sqrt{\aZ}\int_{|p-a|<1}dp\hspace{1mm}\pZ  \frac{f_0(p)}{|p-a|}+ \sqrt{\aZ}\int_{|p-a|>1}dp\hspace{1mm}\pZ  \frac{f_0(p)}{|p-a|}\\
&\hspace{1cm}\lesssim (\aZ)^{1/2}\left(\|f_0\|_{L^\infty_{1}}+\|f_0\|_{L^1_1}\right),
\end{align*}
which yields the desired result.
\end{proof}

Now we prove an estimate of the integral of a solution over a hypersurface.

\begin{lemma}
	\label{lemma4} Let $a^\mu$ be a space-like vector and  $b^\mu$ be an energy-momentum  vector.
Suppose $f_0$ satisfies (\ref{initial}) for $\exm >\frac{5}{2}$. Let $f$ be a solution to (\ref{RBE})  satisfying (\ref{Sol}).
	Then we have
\begin{align*}
\int_{\rth}dp\hspace{1mm}(\pZ )^{1/2}f(p,t)\delta(a^\mu(p_\mu-b_\mu))\lesssim&\frac{1}{|a|}
\|f_0\|_{L^\infty_{\exm}}+\frac{1-e^{-C_{\ell}t}}{C_{\ell}}\frac{1}{\sqrt{a^{\mu}a_{\mu}}} \|f_0\|_{L^1_1}^2.
\end{align*}
where $C_{\ell}>0$ is the constant given in Lemma \ref{lowerL}.
\end{lemma}

We remark again that the proof of Lemma \ref{lemma4} only uses the entropy bounds from \eqref{initial} and \eqref{Sol} in the application of the lower bound for the \textit{collision frequency} from Lemma \ref{lowerL}.

\begin{proof}
We integrate (\ref{recall}) with respect to $(\pZ )^{1/2}\delta(a^\mu(p_\mu-b_\mu))dp$ to obtain
\begin{align*}
&\partial_t\int_{\rth} dp\hspace{1mm}(\pZ )^{1/2}f(p)\delta(a^\mu(p_\mu-b_\mu))+\int_{\rth}dp\hspace{1mm}(\pZ )^{1/2} f(p)Lf(p)\delta(a^\mu(p_\mu-b_\mu))\\
&\hspace{1.5cm}=\int_{\rth} dp\hspace{1mm}(\pZ )^{1/2}Q^+(f,f)\delta(a^\mu(p_\mu-b_\mu)).
\end{align*}
 Then we have from Lemma \ref{lowerL} and Proposition \ref{lemma2}:
 \begin{align*}
&\partial_t\int_{\rth} dp\hspace{1mm}(\pZ )^{1/2}f(p)\delta(a^\mu(p_\mu-b_\mu))+C_{\ell}\int_{\rth}dp\hspace{1mm}(\pZ )^{1/2} f(p)\delta(a^\mu(p_\mu-b_\mu))\cr
&\hspace{1.5cm}\lesssim \frac{1}{\sqrt{a^{\mu}a_{\mu}}}\|f\|^2_{L^1_{1/2}}
\lesssim \frac{1}{\sqrt{a^{\mu}a_{\mu}}}\|f_0\|^2_{L^1_{1}},
\end{align*}
which immediately gives
 $$
 \frac{d}{dt}\left(e^{C_{\ell}t}\int_{\rth}dp ~ (\pZ )^{1/2}f(p)\delta(a^\mu(p_\mu-b_\mu))\right)\lesssim\frac{1}{\sqrt{a^{\mu}a_{\mu}}} \|f_0\|^2_{L^1_{1}} e^{C_{\ell}t}.
 $$
Therefore, we obtain
\begin{multline*}
\int_{\rth}dp\hspace{1mm}(\pZ )^{1/2}f(p)\delta(a^\mu(p_\mu-b_\mu))
\\
\lesssim
e^{-C_{\ell}t}\int_{\rth}dp\hspace{1mm}(\pZ )^{1/2}f_0(p)\delta(a^\mu(p_\mu-b_\mu))+\frac{1-e^{-C_{\ell}t}}{C_{\ell}}\frac{1}{\sqrt{a^{\mu}a_{\mu}}} \|f_0\|^2_{L^1_{1}}.
\end{multline*}
It remains to estimate the first term in the upper bound. For this, we use the standard polar-coordinate representation of $p\mapsto (r,\theta,\phi)$ with $z$ axis parallel to the vector $a=(a^1, a^2, a^3)$ from $a^\mu$
so that we have
\begin{align*}
I_1& \eqdef \int_{\rth}dp\hspace{1mm}(\pZ )^{1/2}f_0(p)\delta(a^\mu(p_\mu-b_\mu))
\\
&=\int_{0}^\infty r^2 dr\   \int_0^{2\pi}d\theta \int_{0}^{\pi}\sin\phi d\phi   (1+r^2)^{1/4}f_0(p)\delta(-\aZ\sqrt{1+r^2}+|a|r\cos \phi-a^\mu b_\mu)\\
&=\int_{0}^\infty dr\  \int_0^{2\pi}d\theta \int_{0}^{\pi} d\phi \sin\phi r^2(1+r^2)^{1/4}f_0(p)\frac{1}{|a|r}\delta\left(\cos\phi-\frac{\aZ\sqrt{1+r^2}+a^\mu b_\mu}{|a|r}\right).
\end{align*}
We then perform another change of variables, $v=\cos\phi$, to compute
\begin{align*}
I_1&=\int_{0}^\infty dr\   \int_0^{2\pi}d\theta \int_{-1}^{1} dv (1+r^2)^{1/4}f_0(p)\frac{r}{|a|}\delta\left(v-\frac{\aZ\sqrt{1+r^2}+a^\mu b_\mu}{|a|r}\right)\cr
&\leq \|f_0\|_{L^\infty_{\exm}} \int_{0}^\infty dr\   \int_0^{2\pi}d\theta \int_{-1}^{1} dv (1+r^2)^{\frac{1}{4}-\frac{s}{2}}\frac{r}{|a|}\delta\left(v-\frac{\aZ\sqrt{1+r^2}+a^\mu b_\mu}{|a|r}\right)\cr
&\leq 2\pi  \|f_0\|_{L^\infty_{\exm}}\int_{0}^\infty dr\frac{(1+r^2)^{\frac{1}{4}-\frac{s}{2}}r}{|a|}\cr
&\lesssim \frac{\|f_0\|_{L^\infty_{\exm}}}{|a|},
\end{align*}
which holds for $\exm >\frac{5}{2}$.  This completes the proof.
\end{proof}

\section{Propagation of the uniform upper bound}\label{pupper}
In this section, we prove the propagation of the uniform $L^\infty$ upper bound for solutions to \eqref{RBE} that is given in Theorem \ref{main}.

\subsection{Uniform bound of $Q^+$}
The last ingredient that we will use in our proof of Theorem \ref{main} is the uniform upper bound for $Q^+$.   To obtain this bound, we will derive the relativistic version of the Carleman representation:

\begin{proposition}{\bf{[Relativistic Carleman representation]}}\label{relativistic Carleman} The relativistic gain operator $Q^+$ has the following alternative representation:
\begin{align*}
	Q^+(f,f)
	=\frac{1}{4\pZ }\int_{\mathbb{R}^3}\frac{dp'}{\pZp } \  f(p')\int_{(p'^\mu-p^\mu)(q'_\mu-p_\mu)=0}\frac{d\pi_{q'}}{\qZp }\frac{s\sigma u(\pZp +\qZp -\pZ )f(q')}{\left|(\pZp -\pZ )\frac{q'}{\qZp }-(p'-p)\right|},
\end{align*}
where  $u(x)$ is defined as in (\ref{u}) and $d\pi_{q'}$ is the surface measure on the relativistic hypersurface:
$$
\left\{q'^{\mu}\in\mathbb{R}^4\,|\,(p'^\mu-p^\mu)(q'_\mu-p_\mu)=0\right\}.
$$
\end{proposition}

\begin{proof}
We recall from  (\ref{Q original}) and (\ref{W}) that $Q^+$ can be written as
	$$
	Q^+(f,f)= \frac{1}{2\pZ }\int_{\mathbb{R}^3}\frac{dq}{{\qZ }}\int_{\mathbb{R}^3}\frac{dq'\hspace{1mm}}{{\qZp }}\int_{\mathbb{R}^3}\frac{dp'\hspace{1mm}}{{\pZp }} s \sigma \delta^{(4)}(p^\mu +q^\mu -p'^\mu -q'^\mu) f(p')f(q').
	$$
We use \eqref{lorentz.invariant.m} with \eqref{u} to raise the 3-fold integral with respect to $dq$ to a 4-fold integral with respect to $dq^\mu$.  In the rest of this proof we will for brevity use the notation $Q^+ = Q^+(f,f)$. Then we have
\begin{multline*}
Q^+
= \frac{1}{2\pZ }\int_{\rfo}dq^\mu \int_{\mathbb{R}^3}\frac{dq'\hspace{1mm}}{{\qZp }}\int_{\mathbb{R}^3}\frac{dp'\hspace{1mm}}{{\pZp }}  s\sigma \delta^{(4)}(p^\mu +q^\mu -p'^\mu -q'^\mu)\delta(q^\mu q_\mu+1)
\\
\times u(\qZ ) f(p')f(q').
\end{multline*}
Now we reduce the integral with respect to $dq^\mu$ by evaluating the 4-dimensional delta function as below:
\begin{multline*}
Q^+= \frac{1}{2\pZ }  \int_{\mathbb{R}^3}\frac{dq'}{\qZp }\hspace{1mm}\int_{\mathbb{R}^3}\frac{dp'}{\pZp }\hspace{1mm}s\sigma \delta\left((p'^\mu+q'^\mu-p^\mu)(p'_\mu+q'_\mu-p_\mu)+1\right)
\\
\times u(\pZp +\qZp -\pZ )f(p')f(q').
\end{multline*}
Now the Lorentz inner product inside the delta function can be expanded as
\begin{align*}
&(p'^\mu+q'^\mu-p^\mu)(p'_\mu+q'_\mu-p_\mu)+1\\
&\quad=2p'^\mu q'_\mu-2p'^\mu p_\mu-2q'^\mu p_\mu +p'^\mu p'_\mu+q'^\mu q'_\mu+p^\mu p_\mu+1\\
&\quad=2p'^\mu q'_\mu-2p'^\mu p_\mu-2q'^\mu p_\mu +2p^\mu p_\mu\\
&\quad=2(p'^\mu-p^\mu)(q'_\mu-p_\mu),
\end{align*}
where we used that these are all energy-momentum vectors as
\[
p'^\mu p'_\mu+q'^\mu q'_\mu+p^\mu p_\mu+1=-2=2p^\mu p_\mu.
\]
Therefore, the gain term $Q^+$ is equal to
\begin{multline}\label{main upper}
Q^+ =
\\
\frac{1}{4\pZ }  \int_{\mathbb{R}^3}\frac{dp'}{\pZp } f(p')\int_{\mathbb{R}^3}\frac{dq'}{\qZp }\hspace{1mm}s\sigma \delta\left((p'^\mu-p^\mu)(q'_\mu-p_\mu)\right) u(\pZp +\qZp -\pZ )f(q'),
\end{multline}
since $\delta(2x)=\delta(x)/2$. We now apply the simple layer formula in the $q'$ variable as
\[
\int_{\mathbb{R}^n}f(x)\delta(g(x))dx=\int_{g(x)=0}\frac{f(x)}{|\nabla_xg(x)|}d\pi_{x}.
\]
Plugging this into \eqref{main upper} completes the proof.
\end{proof}

We will now show that the gain term is uniformly bounded from above under (\ref{initial}) and (\ref{Sol}). For the estimates in the proof of Proposition \ref{upperG}, we will use the  representation \eqref{main upper} without applying the simple layer formula.

\begin{proposition}{\bf[Uniform upper bound of $Q^+$]} \label{upperG}
Suppose $f_0$ satisfies (\ref{initial}) for $\exm >\frac{5}{2}$. Let $f$ be a solution to (\ref{RBE}) satisfying (\ref{Sol}).
 Then, there exists a uniform constant $C_{Q^+} = C_{Q^+}(\|f_0\|_{L^\infty_{\exm}},\|f_0\|_{L^1_1})>0$
 such that
 $$
 Q^+(f,f)\leq C_{Q^+}.
 $$
\end{proposition}

We mention that the proof of Proposition \ref{upperG} only uses the entropy bounds from \eqref{initial} and \eqref{Sol} in the application of the lower bound for the \textit{collision frequency} from Lemma \ref{lowerL} that was used in the proofs of Lemmas \ref{lemma3} and \ref{lemma4}.

\begin{proof}We start with the relativistic Carleman representation in \eqref{main upper}.  Note that $g^2(p^\mu,q^\mu)\leq s(p^\mu,q^\mu)=s(p'^\mu,q'^\mu)\lesssim \pZp \qZp $ and our hypothesis on the collision cross-section $\sigma$ from \eqref{assumption} says that $\sigma\approx g$. Also $u\leq 1$ as in \eqref{u}. Then we obtain
\begin{equation*}
Q^+(f,f)\lesssim \frac{1}{\pZ }  \int_{\mathbb{R}^3}\frac{dp'}{\pZp }\hspace{1mm}\int_{\mathbb{R}^3}\frac{dq'}{\qZp }\hspace{1mm}(\pZp \qZp )^{3/2} \delta\left((p'^\mu-p^\mu)(q'_\mu-p_\mu)\right) f(p')f(q').
\end{equation*}
	Note that $p'^\mu-p^\mu$ is a space-like vector as long as $p'^\mu-p^\mu \ne 0$ since then
$$
(p'^\mu-p^\mu)(p'_\mu-p_\mu)=g^2(p'^\mu,p^\mu)>0.
$$
Then by Lemma \ref{lemma4} with $a^\mu=(p'^\mu-p^\mu)$ and $b^\mu=p^\mu$ (where the role of $q'^\mu$ is that of $p^\mu$ in Lemma \ref{lemma4}) we obtain that
\begin{multline*}
\frac{1}{\pZ }  \int_{\mathbb{R}^3}dp'(\pZp )^{1/2}f(p')\int_{\mathbb{R}^3}dq'(\qZp )^{1/2}f(q')
\delta\left((p'^\mu-p^\mu)(q'_\mu-p_\mu)\right)
\\
\lesssim \frac{1}{\pZ }  \int_{\mathbb{R}^3}dp' \frac{(\pZp )^{1/2}f(p')}{|p'-p| }\|f_0\|_{L^\infty_{\exm}} +\frac{1-e^{-C_{\ell}t}}{C_{\ell}}\frac{\|f_0\|^2_{L^1_{1}}}{\pZ }  \int_{\mathbb{R}^3}dp'\frac{(\pZp )^{1/2}f(p')}{g(p'^\mu,p^\mu)}
\\
\leq \left(\frac{\|f_0\|_{L^\infty_{\exm}}}{\pZ }+\frac{1-e^{-C_{\ell}t}}{C_{\ell}}\frac{\|f_0\|^2_{L^1_{1}}}{\pZ } \right)\int_{\mathbb{R}^3}dp' \frac{(\pZp )^{1/2}f(p')}{g(p'^\mu,p^\mu)},
\end{multline*}
where in the last inequality we used $|p'-p|\geq g(p'^\mu,p^\mu)$ from \eqref{gINEQ}.
Now we use Lemma \ref{lemma3} to obtain that
\begin{align*}
\frac{1}{\pZ }\int_{\mathbb{R}^3}dp'\frac{(\pZp )^{1/2}f(p')}{g(p'^\mu,p^\mu)}\lesssim \frac{1}{\pZ }(\pZ )^{1/2}(\|f_0\|_{L^\infty_{\exm}}+\|f_0\|_{L^1_1}+ \|f\|^2_{L^1_{1}}).
\end{align*}
Therefore, we conclude that $Q^+(f,f) \le (\pZ )^{-\frac{1}{2}}C_{Q^+}(\|f_0\|_{L^\infty_{\exm}},\|f_0\|_{L^1_1})$.
\end{proof}

Now we are ready to prove our main theorem.

\begin{proof}[Proof of Theorem \ref{main}]
Under the assumptions on the initial date in Theorem \ref{main}, we have a global in time solution to (\ref{RBE}), as given in \cite[Theorem 4.2, page 933]{MR3166961}, satisfying \eqref{Sol}.  Therefore, applying  Lemma \ref{lowerL} and Proposition \ref{upperG} to (\ref{RBE}), we obtain that
\begin{align*}
\partial_t f+C_{\ell}f\leq C_{Q_+},
\end{align*}
which directly implies
$$
f(p,t)\leq e^{-C_{\ell}t}f_0(p)+\frac{C_{Q^+}}{C_{\ell}}\left( 1-e^{-C_{\ell}t} \right) <\infty.
$$
Since we have assumed that $f_0\in L^\infty_{\exm}$,  this completes the proof.
\end{proof}

%
%
%

\section{Propagation of the polynomial and the Maxwellian upper bounds}\label{pmaxwell}
In this section, we will prove Theorem \ref{maxwellian upper bound} and Theorem \ref{poly.upper.bound}. For the proof of both theorems, we will use the following Theorem \ref{L1E} as well as Theorem \ref{main}:

\begin{theorem}[Theorem 5.2 of \cite{MR3166961}]\label{L1E}
Let $f_0$ satisfy the assumptions given in \eqref{Existence}.  Then by \cite[Theorem 4.2, page 933]{MR3166961} we have a unique global in time solution $f(p,t)\ge 0$ to \eqref{RBE}.  This solution will further have the propagation of moments as follows:
\begin{enumerate}
	\item  If $f_0$ additionally satisfies  for some $k \ge 1$ that
$$
\int_\rth f_0(p)(\pZ )^k dp<\infty,
$$
then there exists a uniform constant $C>0$ such that the polynomial moment will propagate in time:
$$
\int_\rth f(p,t)(\pZ )^k dp\le C <\infty,\  \forall t\ge 0.
$$
	\item If $f_0$ additionally satisfies for some constant $R_0 > 0$ that
$$\int_\rth f_0(p)e^{R_0\pZ }dp<\infty,$$ then  the exponential moment will propagate in the sense that there exists a
	constant $R = R(f_0, R_0) > 0$ such  that there is uniform constant $C>0$ satisfying
	$$
	\int_\rth f(p,t)e^{R\pZ }dp\le C <\infty,\ \forall t\ge 0.
	$$
\end{enumerate}
\end{theorem}

We note that in general $R=R(f_0,R_0)>0$ in Theorem \ref{L1E} satisfies $R <R_0$.  Now we are ready to prove Theorem \ref{maxwellian upper bound}. The proof contains several steps.

\begin{proof}[Proof of Theorem \ref{maxwellian upper bound}]
	Suppose that $f$ is a solution to \eqref{RBE} with the initial data $f_0$ which satisfies the assumptions of Theorem \ref{maxwellian upper bound} with $R_0>0$.
We clearly have
$$
\int_\rth f_0(p)(\pZ )^{m_0}e^{\bar{R}_0\pZ }dp<\infty, \quad   \bar{R}_0<R_0.
$$
Then, by Theorem \ref{L1E}, there exists $R=R(f_0,\bar{R}_0)>0$ and $C>0$ such that
$$
\int_\rth f(p,t)(\pZ )^{m_0}e^{R\pZ }dp\le C<\infty,\ \forall t\ge 0.
$$
Now, we fix a constant $R_1>0$ which satisfies that $R_1<R$, and we define
$$
h(p,t)\eqdef f(p,t) (\pZ )^{m_0}e^{R_1\pZ }.
$$
We will show that there exists a constant $C>0$ such that $h(p,t)\leq C$ for all $t\ge 0$.

To this end we observe that $h$ satisfies
	\begin{align*}
		\partial_t h&=\partial_t(f(p,t) (\pZ )^{m_0}e^{R_1\pZ })\\
		&=(\pZ )^{m_0}e^{R_1\pZ }Q(f,f)\\
		&=(\pZ )^{m_0}e^{R_1\pZ }(Q^+(f,f)-Q^-(f,f)).
	\end{align*}
	Since $Q^-(f,f)=fLf$, we have
	$$\partial_t h +hLf=(\pZ )^{m_0}e^{R_1\pZ }Q^+(f,f).$$
We use Lemma \ref{lowerL}, with $C_{\ell}>0$, to further obtain that
\begin{equation}\label{hboundedbyexpf}
\partial_t h +C_{\ell}(\pZ) h\leq (\pZ )^{m_0} e^{R_1\pZ }Q^+(f,f).
\end{equation}
Next, we observe that
\begin{align}\label{similar}
\begin{split}
4(\pZp )^2(\qZp )^2&=2(\pZp )^2(\qZp )^2+2(\pZp )^2(\qZp )^2\cr
 &\geq (\pZp )^2+(\qZp )^2+2\pZp \qZp \cr
 &=(\pZp +\qZp )^2\cr
 &=(\pZ +\qZ )^2\cr
 &\geq (\pZ )^2,
\end{split}
\end{align}
to obtain using \eqref{Q center of momentum} that
 \begin{align*}
Q^+(h,h)&=\int_\rth dq\int_{\mathbb{S}^2}d\omega \ \mollrV \sigma(g,\theta)(\qZp )^{m_0}(\pZp )^{m_0} e^{R_1(\pZp +\qZp )}f(q')f(p').\\
		&=\int_\rth dq\int_{\mathbb{S}^2}d\omega \ \mollrV \sigma(g,\theta) (\qZp )^{m_0}(\pZp )^{m_0}e^{R_1(\pZ +\qZ )}f(q')f(p')\\
		&\geq \int_\rth dq\int_{\mathbb{S}^2}d\omega\  \mollrV \sigma(g,\theta) 2^{-m_0}(\pZ )^{m_0}e^{R_1(\pZ +\qZ )}f(q')f(p')\\
			&\geq 2^{-m_0}(\pZ )^{m_0}e^{R_1\pZ }\int_\rth dq\int_{\mathbb{S}^2}d\omega \ \mollrV \sigma(g,\theta)f(q')f(p')\\
		&=2^{-m_0}(\pZ )^{m_0}e^{R_1\pZ }Q^+(f,f).
\end{align*}
Therefore,  using \eqref{hboundedbyexpf}, we conclude that $h$ satisfies
\begin{align}\label{h}
\partial_t h+C_{\ell}(\pZ) h\leq 2^{m_0}Q^+(h,h).
\end{align}
Now we define $h_0\eqdef f_0 (\pZ )^{m_0} e^{R_1\pZ }.$
Then
$|h_0|_{L^\infty_{\exm}}$ is bounded for any $\exm >0$ because
$$
|h_0|_{L^\infty_{\exm}}=\sup_{p\in \rth}|(\pZ )^{\exm+m_0}f_0(p)e^{R_1\pZ }|\lesssim \sup_{p\in \rth}|f_0(p)e^{R_0\pZ }|\eqdef A_1<\infty,
$$
since $R_1<R<R_0$.  We also have that $|h|_{L^1_1}$ is bounded because
$$
|h|_{L^1_1}=\int_\rth dp\ (\pZ )^{1+m_0}f(p,t)e^{R_1\pZ }\lesssim\int_\rth dp\ f(p,t)e^{R\pZ }\eqdef A_2<\infty.
$$
Further $|h_0|_{L^1_1}$ is bounded because $R_1<R<\bar{R}_0$ and
$$
|h_0|_{L^1_1}=\int_\rth dp\ (\pZ )^{1+m_0}f_0(p)e^{R_1\pZ }\lesssim\int_\rth dp\ f_0(p)e^{\bar{R}_0\pZ }\eqdef A_3<\infty.
$$
We note that the constants $A_1$, $A_2$ and $A_3$ are uniform for fixed $R_0$.
Therefore, by Proposition \ref{upperG}, there exists a constant $C'>0$ such that
 $$
 Q^+(h,h)\leq C', \quad C'= C'(A_1,A_2,A_3).
 $$
The results of Proposition \ref{upperG} as above follow from the fact that  Lemma \ref{lemma3} and Lemma \ref{lemma4} and hence Proposition \ref{upperG} remain true following the same proofs even if $h$ only solves the differential inequality \eqref{h} instead of \eqref{RBE}.  Note that $C'$ depends only on $f_0$, $\bar{R}_0$, $R$, and $s$, but not on $R_1$.
 Hence, we have from (\ref{h}) that
 $$
 \partial_t h+C_{\ell}h\leq C',
 $$
 for $C_{\ell}>0$ and $C'>0$. Therefore we obtain
$$
h(p,t)\leq e^{-C_{\ell}t}h_0(p)+\frac{C'}{C_{\ell}}\left( 1-e^{-C_{\ell}t} \right)\leq |h_0|_{L^\infty}+\frac{C'}{C_{\ell}}\eqdef C_1.
$$
Thus we further obtain
$$f(p,t)\leq C_1(\pZ )^{-m_0}e^{-R_1\pZ },$$ for $t>0$.
This completes the proof of Theorem \ref{maxwellian upper bound}.
\end{proof}

Next we prove the uniform polynomial bound given in Theorem \ref{poly.upper.bound}.

\begin{proof}[Proof of Theorem \ref{poly.upper.bound}]
 We will use the same strategy as in the proof of Theorem \ref{maxwellian upper bound}, thus we only give a brief sketch of the differences.  We suppose that $f$ is a solution to \eqref{RBE} with the initial data $f_0$ which satisfies the assumptions of Theorem \ref{poly.upper.bound} with $\exm > \frac{5}{2}$.
Then, by Theorem \ref{L1E}, there exists  $C>0$ such that
$$
\int_\rth f(p,t)(\pZ )^{m_0}dp\le C<\infty,\ \forall t\ge 0.
$$
Now, we define
$$
h(p,t)\eqdef f(p,t) (\pZ )^{m_0}.
$$
We show that there exists a constant $C>0$ such that $h(p,t)\leq C$ for all $t\ge 0$.

As in the proof of Theorem \ref{maxwellian upper bound}, $h$ satisfies
$$
\partial_t h +hLf=(\pZ )^{m_0}Q^+(f,f).
$$
We use Lemma \ref{lowerL}, with $C_{\ell}>0$, to similarly obtain that
\begin{equation}\label{hboundedbyexpf.poly}
\partial_t h +C_{\ell}(\pZ) h\leq (\pZ )^{m_0} Q^+(f,f).
\end{equation}
Next, we recall \eqref{similar} to obtain similarly using \eqref{Q center of momentum} that
 \begin{equation*}
Q^+(h,h)\geq 2^{-m_0}(\pZ )^{m_0}Q^+(f,f).
\end{equation*}
Therefore,  using \eqref{hboundedbyexpf.poly}, we conclude that $h$ satisfies
\begin{align}\label{h.poly}
\partial_t h+C_{\ell}(\pZ) h\leq 2^{m_0}Q^+(h,h).
\end{align}
Now we define $h_0\eqdef f_0 (\pZ )^{m_0}$.
Then
$|h_0|_{L^\infty_{\exm}}$ is bounded for any $\exm>\frac{5}{2}$ because we assume initially that
$
|h_0|_{L^\infty_{\exm}}=|f_0|_{L^\infty_{m_0+\exm}}<\infty.
$
 We also have that $|h|_{L^1_1}$ is bounded because $|h(t)|_{L^1_1}=|f(t)|_{L^1_{m_0+1}} \le \tilde{C}<\infty$ for a uniform constant $\tilde{C}>0$ $\forall t\ge 0$ by using Theorem \ref{L1E} and $|f_0|_{L^1_{m_0+1}} <\infty$.    Similarly $|h_0|_{L^1_1}=|f_0|_{L^1_{m_0+1}} <\infty$.  Therefore, again by Proposition \ref{upperG}, there exists a constant $C'>0$ such that
 $$
 Q^+(h,h)\leq C', \quad C'= C'(\tilde{C},|f_0|_{L^1_{m_0+1}},|f_0|_{L^\infty_{m_0+\exm}}).
 $$
Again the proofs of Proposition \ref{upperG},  Lemma \ref{lemma3}, and Lemma \ref{lemma4} still go through even if $h$ only solves the differential inequality \eqref{h.poly}.
We conclude (\ref{h.poly}) that
 $$
 \partial_t h+C_{\ell}h\leq C'.
 $$
The proof is completed using the same argument in the proof of Theorem \ref{maxwellian upper bound}.
\end{proof}

%
%
%
%
%

\section{H-theorem}\label{htheorem}

In this section, as an application of our $L^{\infty}$ estimates, we provide a proof of the $H$-theorem as in \eqref{entropy.eq} for the solutions to \eqref{RBE} from \cite[Theorem 4.2, page 933]{MR3166961}.  Note that the $H$-theorem for $(\ref{RBE})$ is also established in  \cite{MR3166961} under different assumptions on the solution, based on the argument of \cite{MR711482}.  However our uniform upper bounds established in Theorems \ref{main}, \ref{maxwellian upper bound} and \ref{poly.upper.bound} enable us to prove the $H$-theorem in a much more direct way.  To this end we consider the following approximated problem for \eqref{RBE} for any small $\varepsilon >0$:
\begin{align}\label{RRBE}
\begin{split}
\partial_t f^{\varepsilon}&=Q(f^{\varepsilon},f^{\varepsilon}),\cr
f^{\varepsilon}(0)&=f_{0,\epsilon}(p)\ge 0,
\end{split}
\end{align}
where truncated initial data $f_{0,\epsilon}$ is defined by
\begin{align*}
f_{0,\epsilon}(p)=\min\left(\frac{1}{\varepsilon}, f_0\right){\bf 1}_{|p|<1/\varepsilon}+\varepsilon e^{-\pZ }.
\end{align*}
Here ${\bf 1}_A$ is the standard indicator function of the set $A$.
Since we have from the definition of $f_{0,\epsilon}$ that
\[
\|f_{0,\epsilon}\|_{L^{\infty}_{\exm}} \le C({\epsilon, \exm})<\infty,
\]
for any $\exm >0$, the global existence of $f^{\varepsilon}\ge 0$ conserving the mass, momentum and energy as in \eqref{conlaw} (without assuming the entropy bound) is guaranteed as in \cite[Theorem 3.1]{MR3169776} and \cite[Theorem 4.2]{MR3166961}.  Then we see that $f^{\varepsilon}$ is strictly positive by using the Duhamel formula for \eqref{RRBE} as
\begin{align}\label{lf}
\begin{split}
f^{\varepsilon}(p,t)&=e^{-\int^t_0Lf^{\varepsilon}(s)ds}f_{0,\varepsilon}(p)+\int^t_0e^{-\int^t_sLf^{\varepsilon}(\tau)d\tau}Q^+(f^{\varepsilon},f^{\varepsilon})(s) ds\cr
&\geq e^{-\int^t_0Lf^{\varepsilon}(s)ds}f_{0,\varepsilon}(p)\cr
&\geq \varepsilon e^{-tC_{\varepsilon,T}(\pZ)^{1/2}}e^{-\pZ },
\end{split}
\end{align}
where we used that $\partial_t f^{\varepsilon}=Q^+(f^{\varepsilon},f^{\varepsilon}) - f^{\varepsilon} L f^{\varepsilon}$ and we further used
\begin{multline*}
Lf^{\varepsilon}=\int_{\rth\times \mathbb{S}^2}\mollrV \sigma f^\varepsilon (q)dqd\omega
\lesssim
\int_{\mathbb{R}^3}(\pZ \qZ )^{1/2}f^\varepsilon(q)dq
\\
\lesssim
\|f^{\varepsilon}\|_{L^1_{1/2}}(\pZ)^{1/2}
\lesssim
\|f_{0,\varepsilon}\|_{L^1_{1}}(\pZ)^{1/2},
\end{multline*}
since $\mollrV \lesssim 1$ and $\sigma\approx g\lesssim (\pZ \qZ )^{1/2}$ from \eqref{assumption} and \eqref{g}.

On the other hand, instead of using Lemma \ref{lowerL} which involves the entropy bound in \eqref{Sol}, we alternatively use the following trivial lower bound:
\begin{align}\label{Lfep}
Lf^{\varepsilon}\geq 0,
\end{align}
to get from (\ref{RRBE}) that
\begin{align}\label{vit}
\partial_t f^{\varepsilon}\leq Q^+(f^{\varepsilon},f^{\varepsilon}).
\end{align}
Now following the proofs of Lemma \ref{lemma3} and Lemma \ref{lemma4} and Proposition \ref{upperG} reveals that, even without the entropy bounds using \eqref{Lfep} instead of Lemma \ref{lowerL}, we still obtain the following time dependent bound (instead of the uniform-in-time bound in Proposition \ref{upperG}) for $T>0$ as:
\begin{align}\label{QQ}
Q^+(f^{\varepsilon}, f^{\varepsilon})\leq C_{\varepsilon,T}.
\end{align}
Note that we have not used the $H$-functional, defined in \eqref{entropy.functional}, at all in our arguments in this section thus far. Now, gathering the estimates in (\ref{Lfep}), (\ref{vit}) and (\ref{QQ}), we derive  the following local-in time bound for \eqref{RRBE}:
\begin{align}\label{Uf}
f^{\varepsilon}\leq f_{0,\varepsilon}+TC_{\varepsilon,T}, \quad \forall 0\leq t\leq T.
\end{align}
Now, the lower and upper bound in (\ref{lf}) and (\ref{Uf}) guarantees that the $H$-functional for $H(f^{\varepsilon}(t))$, in \eqref{entropy.functional}, is well-defined, and the standard formal computations for the $H$-theorem are justified, see for example \cite{DeGroot,C-K}, so that we obtain
\begin{align*}
H(f^{\varepsilon})+\int^t_0D(f^{\varepsilon})ds= H(f_{0,\varepsilon}).
\end{align*}
Then, we let $\varepsilon\rightarrow0$ and use  convexity to obtain that \eqref{entropy.eq} indeed holds for our unique solutions to \eqref{RBE}.  

%
%
%
%
%
%

\section{Asymptotic Behavior}\label{asympt}

In this section, we will show that our uniform estimates from the previous sections can be crucially used in the study of the asymptotic behavior of (\ref{RBE}) as in the following theorem.

\begin{theorem}\label{convergence.thm.J} We assume that \eqref{initial} holds with $\exm > \frac{13}{2}$.
We construct a global relativistic Maxwellian \eqref{E:Maxwelliandef}, denoted $J$, that has the same mass, momentum and energy \eqref{conlaw} as $f_0$; this is  explained in Section \ref{SS:Maxwellians}.   Then, the solution $f$ to (\ref{RBE}) satisfying \eqref{Sol} further converges asymptotically to the relativistic Maxwellian:
\[
\lim_{t\rightarrow\infty}\|f(t)-J\|_{L^1}=0.
\]
\end{theorem}

To prove this theorem, we will make use of the general arguments and strategy from \cite[Section 5, page 708]{C-C-L}.  In that paper a general convergence result is presented.  In order to use those arguments we will prove the following lemma.  After that we will explain how to conclude Theorem \ref{convergence.thm.J} from this lemma.

\begin{lemma}\label{lem:CLAIM}
  Suppose there exists a subsequence $\{t_n\}_{n \ge 1}$ such that
  \begin{align*}
  \lim_{n\rightarrow\infty}\|f(t_n)-J\|_{L^1}=0.
  \end{align*}
  Then we have
  \begin{align*}
  \lim_{n\rightarrow\infty}E(t_n)=0,
  \end{align*}
  where $E(t_n)\eqdef A(t_n)+B(t_n)$, and $A$ and $B$ are defined by
  \begin{align*}
  A(t)& \eqdef \|Q^+\big(f(t),f(t)\big)-Q^+(J,J)\|_{L^1},\cr
  B(t)& \eqdef \|f(t)L\big(f(t)-J\big)\|_{L^1}.
  \end{align*}
\end{lemma}

\begin{proof}[Proof of Lemma \ref{lem:CLAIM}]
We will start by proving that $\displaystyle\lim_{n\rightarrow\infty}A(t_n)=0$. To this end we initially show that $Q^+$ is Lipschitz continuous with respect to translation of $p$ in the $L^1$ space as follows. To do this, for some sufficiently large $R>0$, we consider
\begin{align*}
&\|Q^+(f,f)(\cdot+h,t)-Q^+(f,f)(\cdot,t)\|_{L^1}\cr
&=\int_{\rth}|Q^+(f,f)(p+h,t)-Q^+(f,f)(p,t)|dp\cr
&=\int_{|p|\leq R}|Q^+(f,f)(p+h,t)-Q^+(f,f)(p,t)|dp\cr
&\qquad+\int_{|p|> R}|Q^+(f,f)(p+h,t)-Q^+(f,f)(p,t)|dp\cr
&=I+II,
\end{align*}
By the Cauchy-Schwarz inequality, the Parseval identity and the regularizing estimate of $Q^+$ in \cite[Theorem 1.1, page 164]{MR3880739}, we have
\begin{align*}
I&\leq C_R\left(\int_{\mathbb{R}^3}|1-e^{i\xi\cdot h}|^2\big|\widehat{Q^+(f,f)}\big|^2d\xi\right)^{1/2}\cr
&\leq C_R |h|\left(\int_{\mathbb{R}^3}|\xi|^2\big|\widehat{Q^+(f,f)}\big|^2d\xi\right)^{1/2}\cr
&\leq C_R |h|\|f\|_{L^1}\|f\|_{L^2},
\end{align*}
where we observe that $||f||_{L^2} \lesssim ||f||_{L^\infty_{\exm}} \lesssim 1$ for some $\exm > 3/2$. 

On the other hand, thanks to $\mollrV\lesssim 1$ from \eqref{mollerV} and
$$
g(p^\mu, q^\mu) = g(p^{\prime\mu}, q^{\prime\mu}) \lesssim (\pZp \qZp )^{1/2}
$$
as in \eqref{g}, without loss of generality for $|h| \le 1$, we get
\begin{align*}
II&\leq \int_{|p|> R}|Q^+(f,f)(p+h,t)|dp+\int_{|p|> R}|Q^+(f,f)(p,t)|dp\cr
&= \left(\int_{|p-h|> R}+\int_{|p|>R}\right)dp\int_{\rth}dq \int_{\mathbb{S}^2}d\omega\  \mollrV \sigma(g,\omega) f(p')f(q')\cr
&\lesssim  \left(\int_{|p-h|> R}+\int_{|p|>R}\right)dp\int_{\rth}dq \int_{\mathbb{S}^2}d\omega\  gf(p')f(q')\cr
&\lesssim \left(\int_{|p-h|> R}+\int_{|p|>R}\right)dp\int_{\rth}dq \int_{\mathbb{S}^2}d\omega\  (\pZp \qZp )^{\frac{1}{2}} f(p')f(q').
\end{align*}
We then employ the propagation of polynomial decay in Theorem \ref{poly.upper.bound} to find
\begin{align*}
II&\lesssim \left(\int_{|p-h|> R}+\int_{|p|>R}\right)dp\int_{\rth}dq \int_{\mathbb{S}^2}d\omega\  (\pZp \qZp )^{-\exm+\frac{1}{2}} \|f\|^2_{L^{\infty}_{\exm}}
\end{align*}
Using the computation in (\ref{similar}), we observe that
$
4(\pZp \qZp )^2\geq \pZ \qZ.
$
Then since $\exm > \frac{13}{2}$
we can bound $II$ as
\begin{align*}
II&\lesssim\left(\int_{|p-h|> R}+\int_{|p|>R}\right)dp\int_{\rth}dq \int_{\mathbb{S}^2}d\omega\  (\pZ \qZ )^{-\frac{\exm}{2}+\frac{1}{4}} \|f\|^2_{L^{\infty}_{\exm}}\cr
&\lesssim \left(\int_{|p-h|> R}+\int_{|p|>R}\right)dp \ (\pZ )^{-\frac{\exm}{2}+\frac{1}{4}} \|f\|^2_{L^{\infty}_{\exm}}\cr
&\lesssim
\|f\|^2_{L^{\infty}_{\exm}} R^{\frac{1}{4}+3}\left(\frac{1}{1+|R-h|^{\frac{\exm}{2}}} + \frac{1}{1+|R|^{\frac{\exm}{2}}} \right).
\end{align*}
Therefore, we conclude that
\begin{align}\label{strong.equaty}
\|Q^+(f,f)(\cdot+h,t)-Q^+(f,f)(\cdot,t)\|_{L^1}
= \Lambda_I(h) + \Lambda_{II}(h,R).
\end{align}
Here $\Lambda_I(h)$ and $\Lambda_{II}(h,R)$ are defined as the upper bounds of the respective $I$ and $II$ estimates given above.  Then we first send $|h| \to 0$ which implies $\Lambda_I(h) \to 0$.  Second we notice that $\Lambda_{II}(0,R)$ is arbitrarily small for any large and fixed $R>0$ and further $\Lambda_{II}(0,R)\to 0$ as $R \to \infty$.  Therefore we obtain
\begin{align}\label{strong}
\lim_{h\rightarrow 0}\|Q^+(f,f)(\cdot+h,t)-Q^+(f,f)(\cdot,t)\|_{L^1}=0.
\end{align}
This, combined with the boundedness of $\|Q^+(f,f)\|_{L^1}$ gives the strong compactness of $Q^+$ in $L^1$ using standard arguments \cite{MR1284432}.

Here, the boundedness of $||Q^+(f,f)||_{L^1}$ follows since we observe that
\begin{multline*}
\|Q^+(f,f)\|_{L^1}=\left|\int_\rth dp  \int_\rth dq\int_{\mathbb{S}^2}d\omega\  \mollrV gf(p')f(q')\right|\\
=\left|\int_\rth dp \int_\rth dq\int_{\mathbb{S}^2}d\omega\  \frac{g^2\sqrt{s}}{p^0q^0} f(p')f(q')\right|\\
=\left|\int_\rth dp' \int_\rth dq'\int_{\mathbb{S}^2}d\omega\  \frac{g^2\sqrt{s}}{p'^0q'^0} f(p')f(q')\right|\\
\lesssim \int_\rth dp' \int_\rth dq'\int_{\mathbb{S}^2}d\omega\  \frac{g^2\sqrt{s}}{p'^0q'^0} |f(p')||f(q')|\\
\lesssim \int_\rth dp' \int_\rth dq'\int_{\mathbb{S}^2}d\omega\  (p'^0q'^0)^{\frac{1}{2}}|f(p')||f(q')|\lesssim ||f||_{L^1_{\frac{1}{2}}}^2,
\end{multline*}
where the third identity is by the pre-post change of variables $(p,q)\mapsto (p',q')$ (see \cite{MR2765751} for an explanation of this change of variables in this coordinate system), and the last inequality is by $g\lesssim \sqrt{p'^0q'^0}$ and $s=s(p^\mu,q^\mu)=s(p'^\mu,q'^\mu)\lesssim p'^0q'^0.$

On the other hand,  it can be readily verified from the weak convergence assumption of $f(t_n)$ to $J$ and the regularization of the gain term that
\begin{align}\label{weak}
\lim_{n\rightarrow\infty}
\int_{\mathbb{R}^3}\left\{Q^+\big(f(t_n),f(t_n)\big)-Q^+(J,J)\right\}\phi(p)dp=0,
\end{align}
for any $L^{\infty}(\mathbb{R}^3)$ function $\phi$.
Therefore, we conclude from (\ref{strong}) and (\ref{weak}) that
\begin{align*}
\lim_{n\rightarrow\infty}\|Q^+\big(f(t_n),f(t_n)\big)-Q^+(J,J)\|_{L^1}=0.
\end{align*}
This establishes that $\displaystyle\lim_{n\rightarrow\infty}A(t_n)=0$.


We now prove $\lim_{n\rightarrow\infty}B(t_n)=0.$ For this, we divide the integral as
\begin{align*}
\|f(t_n)L(f(t_n)-J)\|_{L^1}
\leq & \int_{\mathbb{R}^3}f(p,t_n)\bigg|\int_{|p-q|\leq R}\mollrV g\left\{ f(q,t_n)-J \right\}dq\bigg|dp\cr
&+\int_{\mathbb{R}^3}f(p,t_n)\bigg|\int_{|p-q|>R}\mollrV g \left\{f(q,t_n)-J\right\} dq\bigg|dp\cr
&= I(t_n)+II(t_n).
\end{align*}
Now, $I$ vanishes by the compactness of $f(t_n)$:
\begin{align*}
\lim_{n\rightarrow\infty}I(t_n)=0.
\end{align*}
Further $II$ can be controlled as
\begin{align*}
II(t_n)&\leq C\int_{\mathbb{R}^3}f(p,t_n)\int_{|p-q|>R}\mollrV g \big(f(q,t_n)+J\big) dpdq\cr
&\leq\frac{C'}{R}\int_{\mathbb{R}^3}f(p,t_n)\int_{\mathbb{R}^3}|p-q|^2\big(f(q,t_n)+J\big) dpdq\cr
&\leq \frac{C'}{R}\left(\| f\|_{L^1_2}+\| f\|_{L^1_2}^2 \right)
\leq \frac{C\left(\| f_0\|_{L^1_2} \right)}{R}.
\end{align*}
where we used $\mollrV\lesssim 1$ from \eqref{mollerV} and $g\leq |p-q|$ from \eqref{gINEQ}, and the propagation of $L^1$ moments from \cite{MR3166961} that we restated in Theorem \ref{L1E}.  We conclude that $II(t_n) \to 0$  as $R\to \infty$ uniformly in $t_n$.  This completes the proof.
\end{proof}

Now we will use Lemma \ref{lem:CLAIM} to give a proof of Theorem \ref{convergence.thm.J}.

\begin{proof}[Proof of Theorem \ref{convergence.thm.J}]
We will not explain in precise detail why Lemma \ref{lem:CLAIM} implies Theorem \ref{convergence.thm.J}.  We can follow a similar argument as in the proof for the part (II) of Theorem 4 in \cite[Section 5, page 708]{C-C-L}.

We first observe that \eqref{Sol}, Lemma \ref{lowerL}, and Proposition \ref{upperG} together imply that
\begin{equation}
\label{eq.5.7}\lim_{|t-s|\rightarrow 0}||f(t)-f(s)||_{L^1}=0, \quad \text{and}\quad \sup_{t\geq 0}||f(t)||_{L^1_1}<\infty.
\end{equation}
To see the above we additionally use that
\begin{equation}\notag
  ||f(t)-f(s)||_{L^1}=\left|\int_s^t d\tau \int_\rth dp \left(Q^+(f,f)-Q^-(f,f)\right)(p,\tau)\right|\lesssim |t-s|.
\end{equation}
The above follows from the boundedness of $Q(f,f)$ in $L^1(\mathbb{R}^3)$.

We now choose and fix any sequence $\{t_n\}_{n\geq 1}\subset [0,\infty)$ satisfying $t_n\rightarrow \infty$ as $n\rightarrow \infty$ and $\sup_{n\geq 1}||f(t_n)||_{L^1_1}<\infty.$ Then there exists a sequence $\{\bar{t}_n\}_{n\geq 1}\subset [t_n,t_n+\delta_n]$ such that $D(f(\bar{t}_n))\leq \delta_n \rightarrow 0$ as $n\rightarrow \infty$ where
\begin{equation}\notag
  \delta_n\eqdef \left(\int_{t_n}^\infty D(f(t))dt+\frac{1}{n}\right)^{1/2},
\end{equation}
since then
$$
\frac{1}{\delta_n}\int_{t_n}^{t_n+\delta_n}D(f(t))dt<\delta_n.
$$
 Therefore, by the $L^1(\mathbb{R}^3)$-compactness of $\{f(\cdot,t)\}_{t\geq 0}$, we can always find a subsequence of $\{t_n,\bar{t}_n\}_{n\geq 1}$ (still denoted by the same notation) and functions $0\leq f_\infty,\bar{f}_\infty\in L^1(\rth)$ such that $f(t_n)\rightarrow f_\infty$ and $f(\bar{t}_n)\rightarrow \bar{f}_\infty$ as $n\rightarrow \infty$ in $L^1(\rth)$. Then we further conclude that $f_\infty=\bar{f}_\infty=J$, since
 \begin{equation}\notag
   0\le D(\bar{f}_\infty)\leq \lim_{n\rightarrow \infty} D(f(\bar{t}_n))=0.
 \end{equation}
 And $D(f)=0$ implies that $f=J$ as in \eqref{E:Maxwelliandef}.  And similarly for $\{t_n\}$ and $f_\infty$.  Therefore, we observe that we have two subsequences $\{t_n\}$ and $\{\bar{t}_n\}$ such that
 \begin{align*}
\lim_{n\rightarrow\infty}\|f(\bar{t}_n)-J\|_{L^1}=\lim_{n\rightarrow\infty}\|f(t_n)-J\|_{L^1}=0.
\end{align*}
Now, using Lemma \ref{lem:CLAIM}, then we have
\begin{equation}
\label{eq.5.18.a}\lim_{n\rightarrow \infty} E(t_n)=\lim_{n\rightarrow \infty} E(\bar{t}_n)=0.
\end{equation}
Note that, using the same proof as \cite[Eq. (5.15) on page 710]{C-C-L} we have
\begin{equation}\notag
  \| Q(f,f)(t)\|_{L^1(\mathbb{R}^3)} \lesssim \sqrt{D(f(t))}.
\end{equation}
Then again  using the same proof as \cite[Eq. (5.16) on page 710]{C-C-L} we have
\begin{equation}\notag
  \| f(t) - J \|_{L^1(\mathbb{R}^3)} \lesssim \frac{1}{L_R}\left(\sqrt{D(f(t))} + E(t) \right) + \frac{1}{R}.
\end{equation}
where $R>0$ is large and $L_R \eqdef \min_{|p| \le R} L(J)(p)>0$ from Lemma \ref{lowerL}.  Further recall from the proof of Lemma \ref{lem:CLAIM} that we also have
\begin{equation}\label{eq.5.18.b}
\lim_{n\rightarrow \infty} D(f(t_n))= \lim_{n\rightarrow \infty} D(f(\bar{t}_n))=0.
\end{equation}
These are the basic estimates that we will use to conclude the proof.

Now choose the original sequence $\{t_n\}$ at the start of this proof to satisfy
\begin{equation}\notag
  \limsup_{t\to\infty} \| f(t) - J \|_{L^1} = \lim_{t\to\infty} \| f(t_n) - J \|_{L^1}
\end{equation}
Then we choose the subsequences as explained previously in this proof.    Therefore, we conclude that \eqref{eq.5.7}, \eqref{eq.5.18.a} and \eqref{eq.5.18.b} together imply that
\begin{equation}\notag
  \limsup_{t\rightarrow\infty}||f(t)-J||_{L^1}=0.
\end{equation}
This completes the proof.
\end{proof}


\noindent{\bf Acknowledgements} J. W. Jang was supported by
the Korean IBS project IBS-R003-D1.  R. M. Strain was partially supported by the NSF grants DMS-1764177 and DMS-1500916 of the USA.   S.-B. Yun is supported by Samsung Science and Technology Foundation under Project Number SSTF-BA1801-02.     

\providecommand{\bysame}{\leavevmode\hbox to3em{\hrulefill}\thinspace}
\providecommand{\href}[2]{#2}


\begin{thebibliography}{10}
\expandafter\ifx\csname arxiv\endcsname\relax
  \def\arxiv#1{\burlalt{http://arxiv.org/abs/#1}{http://arxiv.org/abs/#1}}\fi
\expandafter\ifx\csname doi\endcsname\relax
  \def\doi#1{\burlalt{http://dx.doi.org/#1}{http://dx.doi.org/#1}}\fi
\expandafter\ifx\csname href\endcsname\relax
  \def\href#1#2{#2}\fi
\expandafter\ifx\csname burlalt\endcsname\relax
  \def\burlalt#1#2{\href{#2}{#1}}\fi

\bibitem{1711.06596}
Ricardo Alonso, Irene~M. Gamba, and Maja Taskovi{\'c},
  \emph{Exponentially-tailed regularity and time asymptotic for the homogeneous
  {B}oltzmann equation}, 2017, \arxiv{1711.06596}.

\bibitem{MR1402446}
H{\aa}kan Andr\'{e}asson, \emph{Regularity of the gain term and strong {$L^1$}
  convergence to equilibrium for the relativistic {B}oltzmann equation}, SIAM
  J. Math. Anal. \textbf{27} (1996), no.~5, 1386--1405, \doi{10.1137/0527076}.

\bibitem{MR2102321}
H{\aa}kan Andr\'{e}asson, Simone Calogero, and Reinhard Illner, \emph{On blowup
  for gain-term-only classical and relativistic {B}oltzmann equations}, Math.
  Methods Appl. Sci. \textbf{27} (2004), no.~18, 2231--2240,
  \doi{10.1002/mma.555}.

\bibitem{MR711482}
Leif Arkeryd, \emph{{$L^{\infty}$} estimates for the space-homogeneous
  {B}oltzmann equation}, J. Statist. Phys. \textbf{31} (1983), no.~2, 347--361,
  \doi{10.1007/BF01011586}.

\bibitem{MR3300786}
A.~Bellouquid, J.~Nieto, and L.~Urrutia, \emph{Global existence and asymptotic
  stability near equilibrium for the relativistic {BGK} model}, Nonlinear Anal.
  \textbf{114} (2015), 87--104, \doi{10.1016/j.na.2014.10.020}.

\bibitem{MR2988960}
Abdelghani Bellouquid, Juan Calvo, Juanjo Nieto, and Juan Soler, \emph{On the
  relativistic {BGK}-{B}oltzmann model: asymptotics and hydrodynamics}, J.
  Stat. Phys. \textbf{149} (2012), no.~2, 284--316,
  \doi{10.1007/s10955-012-0600-0}.

\bibitem{MR1478067}
A.~V. Bobylev, \emph{Moment inequalities for the {B}oltzmann equation and
  applications to spatially homogeneous problems}, J. Statist. Phys.
  \textbf{88} (1997), no.~5-6, 1183--1214, \doi{10.1007/BF02732431}.

\bibitem{Cal}
Simone Calogero, \emph{The {N}ewtonian limit of the relativistic {B}oltzmann
  equation}, J. Math. Phys. \textbf{45} (2004), no.~11, 4042--4052,
  \doi{10.1063/1.1793328}.

\bibitem{MR2989691}
Juan Calvo, \emph{On the hyperbolicity and causality of the relativistic
  {E}uler system under the kinetic equation of state}, Commun. Pure Appl. Anal.
  \textbf{12} (2013), no.~3, 1341--1347, \doi{10.3934/cpaa.2013.12.1341}.

\bibitem{Car}
T.~Carleman, \emph{Probl\`emes math\'{e}matiques dans la th\'{e}orie
  cin\'{e}tique des gaz}, Publ. Sci. Inst. Mittag-Leffler. 2, Almqvist \&
  Wiksells Boktryckeri Ab, Uppsala, 1957.

\bibitem{C-C-L}
Eric~A. Carlen, Maria~C. Carvalho, and Xuguang Lu, \emph{On strong convergence
  to equilibrium for the {B}oltzmann equation with soft potentials}, J. Stat.
  Phys. \textbf{135} (2009), no.~4, 681--736, \doi{10.1007/s10955-009-9741-1}.

\bibitem{C}
Carlo Cercignani, \emph{The {B}oltzmann equation and its applications}, Applied
  Mathematical Sciences, vol.~67, Springer-Verlag, New York, 1988,
  \doi{10.1007/978-1-4612-1039-9}.

\bibitem{C-I-P}
Carlo Cercignani, Reinhard Illner, and Mario Pulvirenti, \emph{The mathematical
  theory of dilute gases}, Applied Mathematical Sciences, vol. 106,
  Springer-Verlag, New York, 1994, \doi{10.1007/978-1-4419-8524-8}.

\bibitem{C-K}
Carlo Cercignani and Gilberto~Medeiros Kremer, \emph{The relativistic
  {B}oltzmann equation: theory and applications}, Progress in Mathematical
  Physics, vol.~22, Birkh\"auser Verlag, Basel, 2002,
  \doi{10.1007/978-3-0348-8165-4}.

\bibitem{DeGroot}
S.~R. de~Groot, W.~A. van Leeuwen, and Ch.~G. van Weert, \emph{Relativistic
  {K}inetic {T}heory. principles and applications.}, North-Holland Publishing
  Co., Amsterdam-New York, 1980.

\bibitem{D-L1}
R.~J. DiPerna and P.-L. Lions, \emph{On the {C}auchy problem for {B}oltzmann
  equations: global existence and weak stability}, Ann. of Math. (2)
  \textbf{130} (1989), no.~2, 321--366, \doi{10.2307/1971423}.

\bibitem{D-L2}
\bysame, \emph{Global solutions of {B}oltzmann's equation and the entropy
  inequality}, Arch. Rational Mech. Anal. \textbf{114} (1991), no.~1, 47--55,
  \doi{10.1007/BF00375684}.

\bibitem{D}
Marek Dudy\'{n}ski, \emph{On the linearized relativistic {B}oltzmann equation.
  {II}. {E}xistence of hydrodynamics}, J. Statist. Phys. \textbf{57} (1989),
  no.~1-2, 199--245, \doi{10.1007/BF01023641}.

\bibitem{D-E0}
Marek Dudy\'{n}ski and Maria~L. Ekiel-Je\.{z}ewska, \emph{Causality of the
  linearized relativistic {B}oltzmann equation}, Phys. Rev. Lett. \textbf{55}
  (1985), no.~26, 2831--2834, \doi{10.1103/PhysRevLett.55.2831}.

\bibitem{D-E0er}
\bysame, \emph{Errata: ``{C}ausality of the linearized relativistic {B}oltzmann
  equation''}, Investigaci\'{o}n Oper. \textbf{6} (1985), no.~1, 2228.

\bibitem{D-E3}
\bysame, \emph{On the linearized relativistic {B}oltzmann equation. {I}.
  {E}xistence of solutions}, Comm. Math. Phys. \textbf{115} (1988), no.~4,
  607--629.

\bibitem{D-E2}
\bysame, \emph{Global existence proof for relativistic {B}oltzmann equation},
  J. Statist. Phys. \textbf{66} (1992), no.~3-4, 991--1001,
  \doi{10.1007/BF01055712}.

\bibitem{E-M-V}
Miguel Escobedo, St\'{e}phane Mischler, and Manuel~A. Valle, \emph{Homogeneous
  {B}oltzmann equation in quantum relativistic kinetic theory}, Electronic
  Journal of Differential Equations. Monograph, vol.~4, Southwest Texas State
  University, San Marcos, TX, 2003.

\bibitem{MR2533928}
Irene~M. Gamba, Vladislav Panferov, and Cedric Villani, \emph{Upper
  {M}axwellian bounds for the spatially homogeneous {B}oltzmann equation},
  Arch. Ration. Mech. Anal. \textbf{194} (2009), no.~1, 253--282,
  \doi{10.1007/s00205-009-0250-9}.

\bibitem{GS4}
R.~T. Glassey and W.~A. Strauss, \emph{Asymptotic stability of the relativistic
  {M}axwellian via fourteen moments}, Transport Theory Statist. Phys.
  \textbf{24} (1995), no.~4-5, 657--678, \doi{10.1080/00411459508206020}.

\bibitem{GL1996}
Robert~T. Glassey, \emph{The {C}auchy problem in kinetic theory}, Society for
  Industrial and Applied Mathematics (SIAM), Philadelphia, PA, 1996,
  \doi{10.1137/1.9781611971477}.

\bibitem{GL-Vacuum}
\bysame, \emph{Global solutions to the {C}auchy problem for the relativistic
  {B}oltzmann equation with near-vacuum data}, Comm. Math. Phys. \textbf{264}
  (2006), no.~3, 705--724, \doi{10.1007/s00220-006-1522-y}.

\bibitem{GS3}
Robert~T. Glassey and Walter~A. Strauss, \emph{Asymptotic stability of the
  relativistic {M}axwellian}, Publ. Res. Inst. Math. Sci. \textbf{29} (1993),
  no.~2, 301--347, \doi{10.2977/prims/1195167275}.

\bibitem{Guo-Strain2}
Yan Guo and Robert~M. Strain, \emph{Momentum regularity and stability of the
  relativistic {V}lasov-{M}axwell-{B}oltzmann system}, Comm. Math. Phys.
  \textbf{310} (2012), no.~3, 649--673, \doi{10.1007/s00220-012-1417-z}.

\bibitem{H-R-Yun}
Byung-Hoon Hwang, R.~Tomasso, and Seok-Bae Yun, \emph{On a relativistic {BGK}
  model for polyatomic gases near equilibrium}, preprint, \arxiv{2102.00462}.

\bibitem{1811.10023}
Byung-Hoon Hwang and Seok-Bae Yun, \emph{Anderson-{W}itting model of the
  relativistic {B}oltzmann equation near equilibrium}, 2018,
  \arxiv{arXiv:1811.10023}.

\bibitem{1801.08382}
\bysame, \emph{{S}tationary solutions to the boundary value problem for
  relativistic {B}{G}{K} model in a slab}, 2018, \arxiv{arXiv:1801.08382}.

\bibitem{Jang2016}
Jin~Woo Jang, \emph{{G}lobal classical solutions to the relativistic
  {B}oltzmann equation without angular cut-off}, Ph.D. thesis, University of
  Pennsylvania, 2016, (ProQuest Document ID 1802787346), pp.~1--135.

\bibitem{MR3880739}
Jin~Woo Jang and Seok-Bae Yun, \emph{Gain of regularity for the relativistic
  collision operator}, Appl. Math. Lett. \textbf{90} (2019), 162--169,
  \doi{10.1016/j.aml.2018.11.001}.

\bibitem{MR2378164}
Zhenglu Jiang, \emph{Global existence proof for relativistic {B}oltzmann
  equation with hard interactions}, J. Stat. Phys. \textbf{130} (2008), no.~3,
  535--544, \doi{10.1007/s10955-007-9453-3}.

\bibitem{MR3169776}
Ho~Lee and Alan~D. Rendall, \emph{The spatially homogeneous relativistic
  {B}oltzmann equation with a hard potential}, Comm. Partial Differential
  Equations \textbf{38} (2013), no.~12, 2238--2262,
  \doi{10.1080/03605302.2013.827709}.

\bibitem{MR1284432}
P.-L. Lions, \emph{Compactness in {B}oltzmann's equation via {F}ourier integral
  operators and applications. {I}, {II}}, J. Math. Kyoto Univ. \textbf{34}
  (1994), no.~2, 391--427, 429--461, \doi{10.1215/kjm/1250519017}.

\bibitem{Pennisi_2018}
Sebastiano Pennisi and Tommaso Ruggeri, \emph{A new {BGK} model for
  relativistic kinetic theory of monatomic and polyatomic gases}, Journal of
  Physics: Conference Series \textbf{1035} (2018), 012005,
  \doi{10.1088/1742-6596/1035/1/012005}.

\bibitem{MR2793935}
Jared Speck and Robert~M. Strain, \emph{Hilbert expansion from the {B}oltzmann
  equation to relativistic fluids}, Comm. Math. Phys. \textbf{304} (2011),
  no.~1, 229--280, \doi{10.1007/s00220-011-1207-z}.

\bibitem{St}
Robert~M. Strain, \emph{Asymptotic stability of the relativistic {B}oltzmann
  equation for the soft potentials}, Comm. Math. Phys. \textbf{300} (2010),
  no.~2, 529--597, \doi{10.1007/s00220-010-1129-1}.

\bibitem{MR2679588}
\bysame, \emph{Global {N}ewtonian limit for the relativistic {B}oltzmann
  equation near vacuum}, SIAM J. Math. Anal. \textbf{42} (2010), no.~4,
  1568--1601, \doi{10.1137/090762695}.

\bibitem{MR2765751}
\bysame, \emph{Coordinates in the relativistic {B}oltzmann theory}, Kinet.
  Relat. Models \textbf{4} (2011), no.~1, 345--359,
  \doi{10.3934/krm.2011.4.345}.

\bibitem{Guo-Strain3}
Robert~M. Strain and Yan Guo, \emph{Stability of the relativistic {M}axwellian
  in a collisional plasma}, Comm. Math. Phys. \textbf{251} (2004), no.~2,
  263--320, \doi{10.1007/s00220-004-1151-2}.

\bibitem{StrainTas}
Robert~M. Strain and Maja Taskovi{\'{c}}, \emph{{E}ntropy dissipation estimates
  for the relativistic {L}andau equation, and applications}, J. Funct. Anal.
  (2019), 1--50, \doi{10.1016/j.jfa.2019.04.007}.

\bibitem{1903.05301}
Robert~M. Strain and Zhenfu Wang, \emph{Uniqueness of bounded solutions for the
  homogeneous relativistic {L}andau equation with {C}oulomb interactions},
  Quart. Appl. Math. \textbf{in press} (2019), 1--39, \arxiv{arXiv:1903.05301}.

\bibitem{MR3166961}
Robert~M. Strain and Seok-Bae Yun, \emph{Spatially homogeneous {B}oltzmann
  equation for relativistic particles}, SIAM J. Math. Anal. \textbf{46} (2014),
  no.~1, 917--938, \doi{10.1137/130923531}.

\bibitem{MR2707256}
Robert~Mills Strain, \emph{Some applications of an energy method in collisional
  kinetic theory}, Ph.D. thesis, Brown University, 2005, (ProQuest Document ID
  305028444), pp.~1--200.

\bibitem{Vil02}
C{\'{e}}dric Villani, \emph{A review of mathematical topics in collisional
  kinetic theory}, Handbook of mathematical fluid dynamics, {V}ol. {I},
  North-Holland, Amsterdam, 2002, pp.~71--305,
  \doi{10.1016/S1874-5792(02)80004-0}.

\bibitem{MR1450762}
Bernt Wennberg, \emph{Entropy dissipation and moment production for the
  {B}oltzmann equation}, J. Statist. Phys. \textbf{86} (1997), no.~5-6,
  1053--1066, \doi{10.1007/BF02183613}.

\end{thebibliography}
\end{document}